\documentclass[11pt, reqno]{amsart}
\usepackage{amssymb, amstext, amscd, amsmath, amsthm, tensor}
\usepackage{color}
\usepackage{dsfont}
\usepackage{enumerate}

\usepackage{fullpage}
\usepackage{hyperref}

\usepackage{xy}
\xyoption{all}

\theoremstyle{plain}
\newtheorem{theorem}{Theorem}[section]
\newtheorem{thm}[theorem]{Theorem}
\newtheorem{cor}[theorem]{Corollary}
\newtheorem{prop}[theorem]{Proposition}
\newtheorem{lem}[theorem]{Lemma}
\newtheorem*{theorem*}{Theorem}
\theoremstyle{definition}
\newtheorem{rem}[theorem]{Remark}
\newtheorem{defn}[theorem]{Definition}

\newtheorem{eg}[theorem]{Example}

\newtheorem{notation}[theorem]{Notation}

\newcommand{\bC}{{\mathbb{C}}}

\newcommand{\bN}{{\mathbb{N}}}

\newcommand{\bT}{{\mathbb{T}}}

\newcommand{\bZ}{{\mathbb{Z}}}

  \newcommand{\A}{{\mathcal{A}}}
  
  \newcommand{\C}{{\mathcal{C}}}
  \newcommand{\D}{{\mathcal{D}}}
  \newcommand{\E}{{\mathcal{E}}}
  
  \newcommand{\G}{{\mathcal{G}}}
\renewcommand{\H}{{\mathcal{H}}}

\renewcommand{\O}{{\mathcal{O}}}

  \newcommand{\T}{{\mathcal{T}}}

\newcommand{\fG}{{\mathfrak{G}}}

\newcommand{\fs}{{\mathfrak{s}}}
\newcommand{\ft}{{\mathfrak{t}}}

\newcommand{\upchi}{{\raise.35ex\hbox{\ensuremath{\chi}}}}

\newcommand{\qforal}{\quad\text{for all}\quad}

\newcommand{\Aut}{\operatorname{Aut}}

\newcommand{\id}{{\operatorname{id}}}

\newcommand{\spn}{\operatorname{span}}

\newcommand{\ca}{\mathrm{C}^*}

\newcommand{\mt}{\varnothing}

\newcommand{\ol}{\overline}

\newcommand{\Per}{{\rm Per}}

\newenvironment{enumeratei}{\begin{enumerate}[\upshape (i)]}{\end{enumerate}}
\begin{document}
%%%%%%%%%%%%%%%%%%%%%%%%%%%%%%%%%%%%%%
\title[The ideal structure of self-similar $k$-graph C*-algebras]{The ideal structures of self-similar $k$-graph C*-algebras}
\author[H. Li]{Hui Li}
\address{Hui Li,
Research Center for Operator Algebras and Shanghai Key Laboratory of Pure Mathematics and Mathematical Practice, Department of Mathematics, East China Normal University, 3663 Zhongshan North Road, Putuo District, Shanghai 200062, China}
\email{lihui8605@hotmail.com}
\author[D. Yang]{Dilian Yang}
\address{Dilian Yang,
Department of Mathematics $\&$ Statistics, University of Windsor, Windsor, ON
N9B 3P4, CANADA}
\email{dyang@uwindsor.ca}

\thanks{HL was supported by National Natural Science Foundation of China (Grant No.~11801176), by Research Center for Operator Algebras of East China Normal University, and by Science and Technology Commission of Shanghai Municipality (STCSM) grant No.~18dz2271000.}
\thanks{DY was partially supported by an NSERC Discovery Grant 808235.}

\begin{abstract}
Let $(G, \Lambda)$ be a self-similar $k$-graph with a possibly infinite vertex set $\Lambda^0$. We associate a universal C*-algebra $\O_{G,\Lambda}$ to $(G,\Lambda)$. The main purpose of this paper is to investigate the ideal structures of $\mathcal{O}_{G,\Lambda}$. We prove that there exists a one-to-one correspondence between the set of all $G$-hereditary and $G$-saturated subsets of $\Lambda^0$ and the set of all gauge-invariant and diagonal-invariant ideals of
$\O_{G,\Lambda}$. Under some conditions, we characterize all primitive ideas of $\O_{G,\Lambda}$. Moreover, we describe the Jacobson topology of some concrete examples, which includes the C*-algebra of the product of odometers. On the way to our main results, we study self-similar $P$-graph C*-algebras in depth.
\end{abstract}


%\date{}
\maketitle
%%%%%%%%%%%%%%%%%%%%%%%%%%%%%%%%%%%%%%

\section{Introduction}
\label{S:intro}

Let $\Lambda$ be a row-finite source-free $k$-graph with finite vertices, and $G$ be a (countable discrete) group acting on $\Lambda$. If the action is self-similar, then we associated $(G,\Lambda)$ a universal C*-algebra $\O_{G,\Lambda}$ which is called the \textit{self-similar $k$-graph C*-algebra} of $(G,\Lambda)$ (\cite{LY17_2}). Roughly speaking, $\O_{G,\Lambda}$ is generated by a universal pair $\{u,s\}$ of representations, where $u$ is a unitary representation of $G$ in $\O_{G,\Lambda}$ and $s$ is a Cuntz-Krieger representation of  $\Lambda$ in $\O_{G,\Lambda}$ such that $u$ and $s$ are compatible with respect to the underlying self-similar action of $(G,\Lambda)$. It turns out that those C*-algebras are so broad that they include many known important C*-algebras, such as unital $k$-graph C*-algebras introduced by Kumjian-Pask \cite{KP00}, Exel-Pardo algebras \cite{EP17}, unital Katsura algebras \cite{Kat08_1}, and Nekrashevych algebras \cite{Nek09}. In \cite{LY17_2}, we proved that $\mathcal{O}_{G,\Lambda}$ is always nuclear and satisfies the UCT, completely characterized the simplicity of $\O_{G,\Lambda}$, and further showed that if $\mathcal{O}_{G,\Lambda}$ is simple then $\mathcal{O}_{G,\Lambda}$ is either stably finite or purely infinite. Very recently, under some conditions, in \cite{LY18} we identified the KMS state space of $\mathcal{O}_{G,\Lambda}$ with the tracial state space of the C*-algebra of the periodicity group $\Per_{G,\Lambda}$ of $(G,\Lambda)$.

In this paper, as a natural continuation of \cite{LY17_2, LY18}, we study the ideal structures of self-similar $k$-graph C*-algebras.
In Section~\ref{S:Pre}, some necessary backgrounds are given. In Section~\ref{S:ssP}, we introduce the notion of self-similar $P$-graphs and study their C*-algebras in detail. The purpose of studying such far-reaching generalization is twofold. First, since self-similar $P$-graphs are a natural generalization of self-similar $k$-graphs, we take this opportunity to study their associated C*-algebras. Second, more importantly, they play a vital role in Section~\ref{S:primitive} for the study of the primitive ideals of a self-similar $k$-graph C*-algebra, which is one of our main goals here. 
%%%%%%%%%%%%%%
The reason why it is inevitable to involve self-similar $P$-graph theory here is simply because taking quotient is not closed for self-similar $k$-graph C*-algebras, but lies in the realm of self-similar $P$-graph C*-algebras.
%%%%%%%%%%%%%%
When $P=\bN^k$ and $\Lambda^0$ is finite, the definition of self-similar $P$-graph C*-algebras in this paper coincides with the one given in \cite{LY17_2}. We should also mention that, very recently, Exel, Pardo, and Starling in \cite{EPS18} provide a definition for self-similar directed graph C*-algebras by using a different approach. It turns out that our C*-algebras in the case of $P=\bN$ are isomorphic to theirs. 
For self-similar $k$-graph C*-algebras, we generalize a series of structure theorems from the unital case proved in  \cite{LY17_2} to the nonunital case. 
%%%%%%%%%%%%%%%%%
In Section~\ref{S:gdinv}, we prove that for any self-similar $k$-graph there is a one-to-one correspondence between the set of all $G$-hereditary and $G$-saturated subsets of $\Lambda^0$ and the set of its all gauge-invariant and diagonal-invariant ideals of the self-similar $k$-graph C*-algebra. We also find a new phenomenon,
which does not arise in ordinary $k$-graph C*-algebras.
That is, a gauge-invariant ideal of a self-similar $k$-graph C*-algebra is not necessarily diagonal-invariant. However, under some mild conditions, we are able to show such a phenomenon disappears. 
In Section~\ref{S:primitive}, under some conditions, we characterize all primitive ideas of self-similar $k$-graph C*-algebras. Finally, in Section~\ref{S:ex} we apply our result from Section~\ref{S:primitive} to some concrete examples of self-similar $k$-graphs. 
For those examples, we not only identify their primitive ideals, but also clearly describe the Jacobson topology of their primitive ideal spaces. 
Those examples include (1) self-similar $k$-graphs with underlying $k$-graphs strongly aperiodic (cf.~\cite{KP14}), and (2)  the motivating example of our series of research \cite{LY17, LY17_2, LY18} -- the product of odometers, which has connections to many other areas like number theory.

\subsection*{Notation and Conventions}

Throughout the paper, $k$ is fix a positive integer, which is allowed to be $\infty$.

All semigroups are assumed to be unital; and all topological spaces are assumed to be second countable. 
By an ideal of a C*-algebra, we always mean a closed two-sided ideal. A homomorphism between two C*-algebras is always assumed to be a *-homomorphism.

Denote by $\mathbb{N}$ the set of all nonnegative integers.  Denote by $\{e_i\}_{i=1}^{k}$ the standard basis of $\mathbb{Z}^k$. For $t=(t_1,\ldots, t_k)\in\bT^k$ and $n=(n_1,\ldots, n_k)\in \bZ^k$, let $t^n:=\prod_{i=1}^{k}t_i^{n_i}$. The identity of a group $G$ (resp. a semigroup $P$) is usually denoted by $1_G$ (resp. $1_P$), sometimes by $0_G$ (resp. $0_P$) if $G$ (resp. $P$) is abelian. For a cancellative abelian semigroup $P$, $\fG(P)$ denotes its Grothendieck group. For a C*-algebra $A$, $M(A)$ denotes its multiplier algebra.

\section{Preliminaries}
\label{S:Pre}

\subsection{Hilbert algebras}

In this subsection, we introduce some basics about traces from \cite{Dix77} and Hilbert algebras from \cite{Dix81}.

\begin{defn}
Let $A$ be a C*-algebra and let $\tau:A^+ \to [0,\infty]$ be a map. Then $\tau$ is called a \emph{trace} if
\begin{enumeratei}
\item $\tau(a+b)=\tau(a)+\tau(b)$ for all $a,b \in A^+$;
\item $\tau(\lambda a)=\lambda \tau(a)$ for all $a \in A^+,\lambda \geq 0$;
\item $\tau(a^*a)=\tau(aa^*)$ for all $a \in A$.
\end{enumeratei}
Furthermore, a trace $\tau$ is said to be
\begin{enumeratei}
\item \emph{faithful} if for any $a \in A^+,\phi(a)=0 \implies a=0$;
\item \emph{densely defined} if $\{a \in A^+:\tau(a) <\infty\}$ is dense in $A^+$;
\item \emph{lower semicontinuous} if for any $\lambda \in \mathbb{R},\{a \in A^+:\tau(a) >\lambda\}$ is open;
\item \emph{semifinite} if for any $a \in A^+,\tau(a)=\sup\{\tau(b):0 \leq b \leq a, \tau(b)<\infty\}$.
\end{enumeratei}
\end{defn}

\begin{lem}\label{densely def lower semicont trace is semifin}
Let $A$ be a C*-algebra, and $\tau$ be a densely defined, lower semicontinuous trace on $A$. Then $\tau$ is semifinite.
\end{lem}

\begin{defn}
Let $\mathcal{A}$ be a $*$-algebra endowed with an inner product $\langle\cdot,\cdot\rangle$. Then $\mathcal{A}$ is called a \emph{Hilbert algebra} if
\begin{enumeratei}
\item $\langle x,y\rangle=\langle y^*,x^*\rangle$ for all $x,y \in \mathcal{A}$;
\item $\langle xy,z\rangle=\langle y,x^*z\rangle$ for all $x,y,z \in \mathcal{A}$;
\item for any $x \in \mathcal{A}$, the map $U_x:\mathcal{A} \to \mathcal{A}$ by $U_x(y)=xy$ is continuous;
\item $\spn\mathcal{A}^2$ is dense in $\mathcal{A}$.
\end{enumeratei}
Denote by $\mathcal{H}$ the completion of $\mathcal{A}$ as an inner product space. Then $\mathcal{A}$ embeds in $B(\mathcal{H})$ by $x \mapsto U_x$ and $\mathcal{A}''$ is called the (left) von Neumann algebra associated with $\mathcal{A}$.
\end{defn}

\begin{thm}\label{exist a semifinite trace}
Let $\mathcal{A}$ be a Hilbert algebra. Then there exists a faithful semifinite ultraweakly lower semicontinuous trace on $\mathcal{A}''$ which is finite on $a^*a$ for all $a \in \mathcal{A}$.
\end{thm}

\subsection{Cuntz-Pimsner algebras}

In this subsection we briefly recall the definition of Cuntz-Pimsner algebras of product systems over cancellative abelian semigroups from \cite{Fow02}.

\begin{defn}
Let $P$ be a cancellative abelian semigroup, $A$ be a C*-algebra, and $X_p$ be a nondegenerate C*-correspondence over $A$ such that $\phi_p(A) \subseteq \mathcal{K}(X_p)$ for all $p \in P$. Denote by $X:=\bigsqcup_{p \in P}X_p$. Then $X$ is called a \emph{product system} over $P$ with coefficient $A$ if
\begin{enumeratei}
\item $X$ is a semigroup;
\item $X_{0_P}=A$;
\item $X_p \cdot X_q \subseteq X_{p+q}$ for all $p,q \in P$;
\item for $p,q \in P \setminus \{0_P\}$, there exists an isomorphism from $X_p \otimes_A X_q$ onto $X_{p+q}$ by sending $x\otimes y$ to $xy$ for all $x \in X_p$ and $y \in X_q$;
\item for $p \in P$, the multiplication $X_{0_P} \cdot X_p$ is implemented by the left action of $A$ on $X_p$, and the multiplication $X_p \cdot X_{0_P}$ is implemented by the right action of $A$ on $X_p$.
\end{enumeratei}
\end{defn}

\begin{defn}
\label{D:Toep}
Let $P$ be a cancellative abelian semigroup, $A$ and $B$ be C*-algebras, $X$ be a product system over $P$ with coefficient $A$, and $\psi:X \to B$ be a map.
For $p \in P$, denote by $\psi_p:=\psi \vert_{X_p}$. Then $\psi$ is called a \emph{representation} of $X$ if
\begin{enumeratei}
\item $(\psi_{p},\psi_0)$ is a Toeplitz representation of $X_p$ for all $p \in P$; and
\item $\psi_p(x)\psi_q(y)=\psi_{p+q}(xy)$ for all $p,q \in P,\, x \in X_p,\, y \in X_q$.
\end{enumeratei}
For $p \in P$, denote by $\psi_p^{(1)}:\mathcal{K}(X_p) \to B$ the homomorphism such that $\psi_p^{(1)}(\Theta_{x,y})=\psi_p(x)\psi_p(y)^*$ for all $x,\, y \in X_p$. The representation $\psi$ is said to be
\emph{Cuntz-Pimsner covariant} if $\psi_{0_P}(a)=\psi_p^{(1)}(\phi_p(a))$ for all $p \in P$ and $a \in A$. The \emph{Cuntz-Pimsner algebra} $\mathcal{O}_X$ of $X$ is generated by a Cuntz-Pimsner covariant representation $j_X$ such that for any Cuntz-Pimsner covariant representation $\psi$ of $X$ into a C*-algebra $B$, there is a homomorphism $h: \O_X\to B$ such that $h \circ j_X=\psi$.
\end{defn}

The following result is a combination of the gauge-invariant uniqueness theorem \cite[Corollary~4.14]{CLSV11} and the observation \cite[Proposition~2.8]{LY17_2}.

\begin{thm}\label{T:GIUCLSV}
Let $X$ be a product system over $\mathbb{N}^k$ with coefficient $A$, and $\psi$ be a Cuntz-Pimsner covariant representation of $X$. Denote by
$h:\mathcal{O}_X \to C^*(\psi(X))$ the homomorphism induced from the universal property of $\mathcal{O}_X$. Suppose that there exists a strongly continuous homomorphism $\alpha:\mathbb{T}^k \to \mathrm{Aut}(C^*(\psi(X)))$ such that $\alpha_t(\psi_p(x))=t^p \psi_p(x)$ for all $t \in \mathbb{T}^k,\, p \in \mathbb{N}^k,\, x \in X_p$. Moreover, suppose that $\psi_{0_P}$ is injective. Then $h$ is an isomorphism.
\end{thm}

\subsection{Groupoid C*-algebras}

In this subsection we recap the background of groupoid C*-algebras from \cite{Ren80}.

A \emph{groupoid} is a small category with inverses. A \emph{topological groupoid} is a topological space and a groupoid such that the product and inverse maps are both continuous. Let $\Gamma$ be a topological groupoid. A subset $A \subseteq \Gamma$ is called a \emph{bisection} if the restrictions $r \vert_A$ and $s \vert_A$ are both homeomorphisms onto their images. An \emph{ample groupoid} is a Hausdorff topological groupoid with an open base consisting of compact open bisections.

Let $\Gamma$ be an ample groupoid. For $u \in \Gamma^0$, define $\Gamma^u:=r^{-1}(u)$, define $\Gamma_u:=s^{-1}(u)$, and define the \emph{isotropy group} at $u$ by $\Gamma_u^u:=r^{-1}(u) \cap s^{-1}(u)$. Then $\Gamma$ is said to be \emph{topologically principal} if the set of units whose isotropy groups are trivial is dense in $\Gamma^{0}$. Moreover, a subset $E \subseteq \Gamma^{0}$ is said to be \emph{invariant} if $r(s^{-1}(E)) \subseteq E$. Then $\Gamma$ is said to be \emph{minimal} if the only open invariant subsets of $\Gamma^0$ are $\mt$ and $\Gamma^{0}$.

Let $\Gamma$ be an ample groupoid. For $f,g \in C_c(\Gamma)$, define
\[
\Vert f \Vert_I:=\max\Big\{\sup_{u \in \Gamma^{0}}\sum_{\gamma \in r^{-1}(u)}\vert f(\gamma) \vert, \sup_{u \in \Gamma^0}\sum_{\gamma \in s^{-1}(u)}\vert f(\gamma) \vert \Big\};
\]
\[
(f* g) (\gamma):=\sum_{s(\beta)=r(\gamma)}f(\beta^{-1})g(\beta\gamma); \text{ and } f^*(\gamma):=\overline{f(\gamma^{-1})}.
\]
Then $\Vert\cdot\Vert_I$ is a norm, which is called the \emph{$I$-norm}, and $C_c(\Gamma)$ is a $*$-algebra. A $*$-representation $\pi$ of $C_c(\Gamma)$ on a Hilbert space is said to be \emph{bounded} if $\Vert \pi(f) \Vert \leq \Vert f \Vert_I$ for all $f \in C_c(\Gamma)$. For $u \in \Gamma^0$, define $L^u:C_c(\Gamma) \to B(\ell^2(\Gamma_u))$ by
\[
L^u(f)(\delta_{\gamma}):=\sum_{s(\beta)=r(\gamma)}f(\beta)\delta_{\beta\gamma} \text{ for all } f \in C_c(\Gamma)\text{ and }\gamma \in \Gamma.
\]
$L^u$ is a bounded $*$-representation called the \emph{left regular representation} at $u$. Define the \emph{left regular representation} $L:=\bigoplus_{u \in \Gamma^0}L^u$. For $f \in C_c(\Gamma)$, define $\Vert f \Vert_r:=\Vert L(f)\Vert$. Then $\Vert\cdot\Vert_r$ is a C*-norm on $C_c(\Gamma)$. The completion of $C_c(\Gamma)$ under the $\Vert\cdot\Vert_r$-norm is called the \emph{reduced groupoid C*-algebra} of $\Gamma$, which is denoted by $C_r^*(\Gamma)$. Define $\Vert f \Vert:=\sup_\pi \Vert \pi(f) \Vert$, where $\pi$ runs through all bounded $*$-representations of $C_c(\Gamma)$. Then $\Vert\cdot\Vert$ is a C*-norm on $C_c(\Gamma)$, and the completion of $C_c(\Gamma)$ under the $\Vert\cdot\Vert$-norm is called the \emph{full groupoid C*-algebra}, denoted as $\ca(\Gamma)$.

\begin{thm}[{\cite[Theorem~4.4]{Exe11}}]\label{CK thm gpoid}
Let $\Gamma$ be a topologically principal ample groupoid, and $I$ be a nonzero ideal of $\ca_r(\Gamma)$. Then $C_0(\Gamma^0) \cap I$ is nonzero.
\end{thm}

\begin{defn}[{\cite[Definitions~5.1, 5.4]{RS17}}]
Let $\Gamma$ be an ample groupoid. Define an equivalence $\sim_\Gamma$ on $C_c(\Gamma^0,\mathbb{N})$ as follows: for $f,g \in C_c(\Gamma^0,\mathbb{N})$, define $f \sim_\Gamma g$ if there exist compact open bisections $E_1,\dots,E_n$ of $\Gamma$ such that $f=\sum_{i=1}^{n}\mathds{1}_{s(E_i)}\text{ and } g=\sum_{i=1}^{n}\mathds{1}_{r(E_i)}$. Define an abelian semigroup by $S(\Gamma):=C_c(\Gamma^0,\mathbb{N})/\!\!\sim_\Gamma$ which is called the \emph{type semigroup} of $\Gamma$. The equivalence class of $f \in C_c(\Gamma^0,\mathbb{N})$ in $S(\Gamma)$ is written as $[f]$. Moreover, define a preorder on $S(\Gamma)$ as follows: for any $s,t \in S(\Gamma)$, define $s \leq t$ if there exists $r \in S(\Gamma)$ such that $s+r=t$. Furthermore, $S(\Gamma)$ is said to be \emph{almost unperforated} if for any $s,t \in S(\Gamma)$, any $n,m \geq 0$, we have $ns \leq mt$ and  $n >m \implies s \leq t$.
\end{defn}

\begin{thm}[{\cite[Theorem~5.1]{BCFS14}, \cite[Corollary~6.6 and Theorem~7.4]{RS17}}]
\label{T:RS}
Let $\Gamma$ be an ample groupoid such that $\Gamma^0$ is not compact and $C^*(\Gamma)$ is simple. Then $\ca(\Gamma)$ is stably finite if and only if there exists a faithful semifinite trace $\tau$ on $\ca(\Gamma)$ such that $0<\tau(\mathds{1}_K)<\infty$ for all compact open subset $K \subseteq \Gamma^0$. Furthermore, suppose that $S(\Gamma)$ is almost unperforated. Then $\ca(\Gamma)$ is either stably finite or purely infinite.
\end{thm}

%%%%%
\subsection{$P$-graph C*-algebras}

This subsection provides a brief introduction for $P$-graph C*-algebras. The main sources are \cite{CKSS14, Yan15}.

%\begin{notation}
%\end{notation}

\begin{defn}
Let $P$ be a cancellative abelian semigroup. A countable small category $\Lambda$ is called a \emph{$P$-graph} if there exists a functor $d:\Lambda \to P$ satisfying that for $\mu\in\Lambda, p,q \in P$ with $d(\mu)=p+q$, there exist unique $\alpha,\beta\in \Lambda$ such that $\mu=\alpha\beta$ with $d(\alpha)=p, d(\beta)=q$.

An $\bN^k$-graph is also known as a \emph{$k$-graph}.
\end{defn}

%\begin{notation}
%Let $P$ be a cancellative abelian semigroup, and
Let $\Lambda$ be a $P$-graph.  For $A,B \subseteq \Lambda$, denote by $AB:=\{\mu\nu:\mu \in A,\nu\in B,s(\mu)=r(\nu)\}$, and  for $p \in P$, denote by $\Lambda^p:=d^{-1}(p)$.
%\end{notation}

\begin{defn}
Let $\Lambda$ be a $P$-graph. Then $\Lambda$ is said to be \emph{row-finite} if $\vert v\Lambda^{p}\vert<\infty$ for all $v \in \Lambda^0$ and $p \in P$. $\Lambda$ is said to be \emph{source-free} if $v\Lambda^{p} \neq \mt$ for all $v \in \Lambda^0$ and $p \in P$.
\end{defn}

\begin{defn}
Let $P$ be a countable cancellative abelian semigroup. Define $\Omega_P:=\{(p,q) \in P \times P:q-p \in P\}$, define $\Omega_P^0:=\{(p,p):p \in P\}$, for $(p,q), (q,m) \in \Omega_P$, define
$r(p,q):=(p,p)$, $s(p,q):=(q,q)$, $(p,q)(q,m):=(p,m),d(p,q):=q-p$. Then $\Omega_P$ is a row-finite and source-free $P$-graph. Let $\Lambda$ be a $P$-graph. An \emph{infinite path} of $\Lambda$ is a functor from $\Omega_P$ to $\Lambda$. The set of all infinite paths of $\Lambda$ is denoted by $\Lambda^\infty$.
\end{defn}

\subsection*{Standing assumptions} \textsf{Throughout the rest of this paper, whenever we deal with $P$-graphs for some cancellative abelian semigroup $P$, $P$ and its Grothendieck group $\fG(P)$ are always considered to be discrete. Moreover, all $P$-graphs are assumed to be row-finite and source-free.}

\begin{defn}[{\cite[Definition~3.5]{PRS08}}]
Let $\Lambda$ be a $k$-graph. A function $\ft:\Lambda^0 \to [0,\infty)$ is called a \emph{graph trace} if $\ft(v)=\sum_{\mu \in v \Lambda^p}\ft(s(\mu))$ for all $v \in \Lambda^0$ and $p \in \mathbb{N}^k$. A graph trace $\ft$ is said to be \emph{faithful} if $\ft(v) \neq 0$ for all $v \in \Lambda^0$.
\end{defn}

\begin{defn}[{\cite[Definition~4.1]{CKSS14}}]
\label{D:MT}
Let $\Lambda$ be a $k$-graph. A nonempty subset $T$ of $\Lambda^0$ is called a \emph{maximal tail of $\Lambda$} if
\begin{enumeratei}
\item for any $v \in T,w \in \Lambda^0$, we have $w \Lambda v \neq \mt \implies w \in T$;
\item for any $v \in T, p \in \mathbb{N}^k$, we have $v \Lambda^p T \neq \mt$;
\item for any $v_1,v_2 \in T$, there exists $w \in T$ such that both $v_1 \Lambda w$ and $v_2 \Lambda w$ are non-empty.
\end{enumeratei}
\end{defn}

\begin{defn}
Let $\Lambda$ be a $P$-graph. Then the \emph{$P$-graph C*-algebra} $\O_\Lambda$ is defined to be the universal C*-algebra generated by a family of partial isometries $\{s_\lambda:\lambda\in\Lambda\}$ (known as the universal \emph{Cuntz-Krieger $\Lambda$-family}) satisfying
\begin{enumeratei}
\item $\{s_v\}_{v \in \Lambda^0}$ is a family of mutually orthogonal projections;
\item $s_{\mu\nu}=s_{\mu} s_{\nu}$ if $s(\mu)=r(\nu)$;
\item $s_{\mu}^* s_{\mu}=s_{s(\mu)}$ for all $\mu \in \Lambda$; and
\item $s_v=\sum_{\mu \in v \Lambda^{p}}s_\mu s_\mu^*$ for all $v \in \Lambda^0$ and $p \in P$.
\end{enumeratei}
Moreover, the subalgebra $\D_\Lambda:=\ol{\spn}\{s_\mu s_\mu^*:\mu\in\Lambda\}$ is called the \emph{diagonal of $\O_\Lambda$}.
\end{defn}

%%%%%%%%%%%%%%%%%%%

\section{Self-similar $P$-graph C*-algebras}
\label{S:ssP}

In this section, we introduce the notion of self-similar $P$-graphs, construct the corresponding self-similar $P$-graph C*-algebras, and prove some structure theorems for those C*-algebras.

\subsection{Self-similar $P$-graphs}

In \cite{LY17_2}, we provided the notion of self-similar $k$-graphs. In this subsection, we generalize it to self-similar $P$-graphs.

%%%%%%%%%%%
\begin{defn}
\label{D:ssP}
Let $\Lambda$ be a $P$-graph, $G$ be a (countable discrete) group acting on $\Lambda$, and $\vert:G\times \Lambda \to G$ be a map. Then $(G,\Lambda)$ is called a \emph{self-similar $P$-graph} if the following properties hold:
\begin{enumeratei}
\item
$G \cdot \Lambda^p \subseteq \Lambda^p$ for all $p \in P$;
\item $s(g \cdot \mu)=g \cdot s(\mu)$ and $r(g \cdot \mu)=g \cdot r(\mu)$ for all $g \in G$ and $\mu \in \Lambda$;
%\item\textcolor{red}{$g \vert_\mu \cdot s(\nu)=g \cdot s(\nu)$ for all $g \in G,\mu,\nu \in \Lambda$ with $s(\mu)=r(\nu)$; This is not needed since Condition 2 implicitly yields}
\item\label{g cdot (mu nu)}
$g\cdot (\mu\nu)=(g \cdot \mu)(g \vert_\mu \cdot \nu)$ for all $g \in G$, $\mu$, $\nu \in \Lambda$ with $s(\mu)=r(\nu)$;

\item\label{g vert_v=g}
$g \vert_v =g$ for all $g \in G$ and $v \in \Lambda^0$;

\item
$g \vert_{\mu\nu}=g \vert_\mu \vert_\nu$ for all $g \in G$, $\mu$, $\nu \in \Lambda$ with $s(\mu)=r(\nu)$;

\item
$1_G \vert_{\mu}=1_G$ for all $\mu \in \Lambda$;

\item\label{(gh) vert_mu}
$(gh)\vert_\mu=g \vert_{h \cdot \mu} h \vert_\mu$ for all $g$, $h \in G$ and $\mu \in \Lambda$.
\end{enumeratei}
\end{defn}

To simplify our writing, let us introduce the following notation. 
\begin{notation}
\label{N:lgl}
For $g\in G$, let $\Lambda\tensor[_g]{\times}{_s}\Lambda:=\{(\mu, \nu)\in \Lambda\times \Lambda:  s(\mu)= g\cdot s(\nu)\}$.
\end{notation}

\begin{defn}
Let $(G,\Lambda)$ be a self-similar $P$-graph. Then $(G,\Lambda)$ is said to be
\begin{enumeratei}
\item \emph{pseudo free} if for any $g\in G$ and $\mu\in\Lambda$, we have $g \cdot \mu=\mu,g \vert_\mu=1_G \implies g=1_G$;

\item \emph{locally faithful} if, for any $g \in G$ and $v\in \Lambda^0$, $g \cdot \mu = \mu$ for all $\mu\in v\Lambda\implies g=1_G$;

\item \emph{strongly locally faithful} if, for any $1_G \neq g \in G$ and $v \in \Lambda^0$, there exists $p \in P$, such that $g \cdot \mu \neq \mu$ for all $\mu\in v\Lambda^p$;

\item \emph{cofinal} if, for any $x \in \Lambda^\infty$ and $v \in \Lambda^0$, there exist $p \in P$, $\mu \in \Lambda$, $g \in G$ such that $s(\mu)=x(p,p)$ and $g \cdot v=r(\mu)$;

\item \emph{aperiodic} if, for any $v \in \Lambda^0$ there exists $x \in v\Lambda^\infty$ such that 
for $g \in G$, $p$, $q \in P$, $\sigma^{p}(x) \neq g \cdot \sigma^q(x)$ provided $g \neq 1_G$ or $p \neq q$;

\item \textit{periodic} if $(G,\Lambda)$ is not aperiodic. 
\end{enumeratei}
\end{defn}

\begin{rem}
As \cite[Remark~2.4]{LY18}, if $(G,\Lambda)$ is a locally faithful $P$-graph, then Conditions~(\ref{g vert_v=g})-(\ref{(gh) vert_mu}) of Definition~\ref{D:ssP} are redundant. 
\end{rem}

\subsection*{Standing assumptions} \textsf{Throughout the rest of this paper, all self-similar $P$-graphs are assumed to be pseudo free.}

\begin{defn}
Let $(G,\Lambda)$ be a self-similar $P$-graph.  For $g \in G$ and $(\mu,\nu)\in \Lambda\tensor[_g]{\times}{_s}\Lambda$, if $\mu(g \cdot x)=\nu x$ for all $x \in s(\nu)\Lambda^\infty$, then $(\mu,g,\nu)$ is called a \emph{cycline triple}.
Cycline triples of the form $(\mu,1_G,\nu)$ are simply called \emph{cycline pairs}.
\end{defn}

Clearly, $(\mu, 1_G, \nu)$  is a cycline pair of $(G, \Lambda)$ if and only if $(\mu,\nu)$ is a cycline pair of the underlying $P$-graph $\Lambda$ in the sense of \cite{Yan15}.

Let $\C_{G,\Lambda}$ (resp. $\C_\Lambda$) be the set of all cycline triples (resp. cycline pairs) of $(G,\Lambda)$ (resp. $\Lambda$). For $p,q\in P$, let
$\C_{G,\Lambda}^{p,q}=\{(\mu, g, \nu)\in \C_{G,\Lambda}: d(\mu)=p, d(\nu)=q\}$, and $\C_{\Lambda}^{p,q}=\{(\mu,\nu)\in\C_\Lambda: d(\mu)=p, d(\nu)=q\}$.

The following characterization of aperiodicity for $(G,\Lambda)$ can be proved completely similar to \cite[Proposition~3.5]{LY18}, and so we skip its proof here.

\begin{prop}\label{P:G-ape}
Let $(G,\Lambda)$ be a self-similar $P$-graph. Then $(G,\Lambda)$ is aperiodic if and only if for any $\mu \in \Lambda,\, g \in G,\, p,\, q \in P$ with $g \neq 1_G$ or $p \neq q$, we have
\begin{enumeratei}
\item
if $p \neq q$ then there exists $\nu \in s(\mu)\Lambda$ such that $d(\nu)- p - q \in P$ and $\nu(p,d(\nu)- q) \neq g \vert_{(\mu\nu)(q,d(\mu)+q)} \cdot \nu(q,d(\nu)-p)$;
\item
if $p=q$, then for any $\nu \in s(\mu)\Lambda$ satisfying that $d(\nu) - p \in P$ and that $g \vert_{(\mu\nu)(p,d(\mu)+p)} \neq 1_G$, there exists $\gamma \in s(\nu)\Lambda$ such that $\nu(p,d(\nu))\gamma \neq g \vert_{(\mu\nu)(p,d(\mu)+p)} \cdot (\nu(p,d(\nu))\gamma)$.
\end{enumeratei}
\end{prop}

%%%%%%%%%%%ADDED
\subsection{Self-similar $P$-graph C*-algebras}

In \cite{LY17_2}, we associated a C*-algebra to each self-similar $k$-graph whose underlying $k$-graph has finite vertices. In this subsection, we generalize it in two ways: (i) the underlying $k$-graph is replaced by a $P$-graph, and (ii) the vertex set of the $P$-graph is not required to be finite.

%The following definition is a generalization of \cite[Definition~3.8]{LY17_2}.
\begin{defn}
\label{D:ssfamily}
Let $(G,\Lambda)$ be a self-similar $P$-graph, and $\A$ be a unital C*-algebra. Suppose that $\{U_g:g\in G\}$ is a family of unitaries in $\A$ and $\{S_\mu: \mu\in\Lambda\}$ is a Cuntz-Krieger $\Lambda$-family in $\A$. If they satisfy the following two properties:
\begin{enumeratei}

\item
$U_{gh}=U_g U_h$ for all $g,h \in G$, and

\item
$U_g S_\mu=S_{g \cdot \mu} U_{g \vert_\mu}$ for all $g \in G,\mu \in \Lambda$,
\end{enumeratei}
then $\{U, S\}$ is called a \textit{self-similar $(G,\Lambda)$-family}.

The C*-algebra generated by a self-similar $(G,\Lambda)$-family $\{U,S\}$ is written as $\ca(U,S)$. 
\end{defn}

%%%%%%%%%%%%%%
\begin{rem}
\label{R:ssfamily}
(i) It directly follows from the definition that a self-similar $(G,\Lambda)$-family $\{U,S\}$ in $\A$ consists of a unitary representation $U$ of $G$ on $\A$ and a Cuntz-Krieger representation $S$ of $\Lambda$ on $\A$, which are compatible with the underlying self-similar action in the sense of Definition~\ref{D:ssfamily} (ii).

(ii)  
%If $|\Lambda^0|<\infty$, as silently done in \cite{LY17_2} (also cf.~\cite{LY17}),  we can always assume that $\sum_{v\in\Lambda^0} S_v=1_\A$. Otherwise, let $P:=\sum_{v\in \Lambda^0}S_v$, %and replace $\A$ by $P\A P$. Then one can check that $\{PUP, S\}$ is a $(G,\Lambda)$-family in $P\A P$, where $(PUP)_g:=P U_g P$ for all $g\in G$. Indeed, to see this,
%from Condition (ii) one has $U_g S_v = S_{g\cdot v} U_{g|_v}=S_{g\cdot v} U_g\Rightarrow U_g (\sum_{v\in \Lambda^0} S_v) =(\sum_{v\in\Lambda^0}S_{g\cdot v}) U_g\Rightarrow U_g P= P %U_g$ for every $g\in G$. 
%%%%
If $|\Lambda^0|<\infty$, as silently done in \cite{LY17_2} (also cf.~\cite{LY17}),  we can always assume that $\sum_{v\in\Lambda^0} S_v=1_\A$. Otherwise, let $P:=\sum_{v\in \Lambda^0}S_v$. 
Then, for every $g\in G$, $P$ and $U_g$ commute since 
 \[
\text{Definition~\ref{D:ssfamily} (ii)}\Rightarrow U_g S_v = S_{g\cdot v} U_{g|_v}=S_{g\cdot v} U_g\Rightarrow U_g \big(\sum_{v\in \Lambda^0} S_v\big) =\big(\sum_{v\in\Lambda^0}S_{g\cdot v}\big) U_g\Rightarrow U_g P= P U_g. 
 \]
One now can easily check that $\{PU, S\}$ is a $(G,\Lambda)$-family in the C*-algebra $P\A P$. % (with the identity $P$).
\end{rem}
%%%%%%%%%%%%%%

\begin{defn}\label{D:O_G,Lambda}
Let $(G,\Lambda)$ be a self-similar $P$-graph. Denote by $\ca_u(G,\Lambda)$ the universal unital C*-algebra generated by a self-similar $(G,\Lambda)$-family $\{u,s\}$.
The \emph{self-similar $P$-graph C*-algebra} of $(G,\Lambda)$, denoted by $\O_{G,\Lambda}$, is a subalgebra of $\ca_u(G,\Lambda)$ defined as 
\[
\O_{G,\Lambda}:=\overline{\mathrm{span}}\{s_\mu u_g s_\nu^*:  g \in G, (\mu,\nu)\in\Lambda\tensor[_g]{\times}{_s} \Lambda\}.
\]
\end{defn}

\subsubsection{The non-triviality of $\O_{G,\Lambda}$ (Cf.~\cite[Remark~3.9]{LY17_2})}
\label{SSS:nontriv}

For $g \in G$, $\mu \in \Lambda$, and $x \in \Lambda^\infty$, define $S_\mu$ and $U_g$ in $B(\ell^2(\Lambda^\infty))$ as follows:
\begin{align*}
U_g(\delta_x):=\delta_{g \cdot x}
\hskip .5cm \text{and }\quad
S_\mu(\delta_x):=\begin{cases}
   \delta_{\mu x}   &\text{ if $s(\mu)=x(0,0)$} \\
   0   &\text{ otherwise}.
\end{cases}
\end{align*}
It is easy to verify that $\{U, S\}$ is a self-similar $(G,\Lambda)$-family in $B(\ell^2(\Lambda^\infty))$. Observe that $S_{\mu}$ and $U_g$ are nonzero for all $\mu \in \Lambda^0$ and $g \in G$. So $u_g,\, s_\mu \neq 0$ for all $G\in G$ and $\mu \in \Lambda$. Therefore $\ca_u(G,\Lambda)$ exists nontrivially. Then one can now easily check that  $\O_{G,\Lambda}$ exists nontrivially.

\subsubsection{The universal property of $\O_{G,\Lambda}$}
\label{SSS:universal}

By virtue of the universal property of $\ca_u(G,\Lambda)$, we obtain the following universal property of $\O_{G,\Lambda}$: given a self-similar $(G,\Lambda)$-family $\{U,S\}$ in a C*-algebra $\A$, there is always a homomorphism $\pi$ from $\O_{G,\Lambda}$ to $\A$ such that $\pi(s_\mu u_g s_\nu^*)=S_\mu U_g S_\nu^*$ for all $g\in G$ and $(\mu, \nu)\in \Lambda\tensor[_g]{\times}{_s}\Lambda$.
In fact, let $\tilde \pi: \ca_u(G,\Lambda)\to \A$ be the homomorphism induced from the universal property of $\ca_u(G,\Lambda)$.
Then $\pi:=\tilde \pi|_{\O_{G,\Lambda}}$. Furthermore, it is easy to see that the range of $\pi$ is $\ca(U,S)$.
It is not surprising that the universal property of $\O_{G,\Lambda}$ will be used frequently later.

\subsubsection{The gauge action of $\O_{G,\Lambda}$}
\label{SSS:gauge}

By the universal property of $\O_{G,\Lambda}$ obtained above, there exists a strongly continuous homomorphism, which is called the \emph{gauge action},
$\gamma:\widehat{\fG(P)} \to \Aut(\O_{G,\Lambda})$ such that
\[
\gamma_f(s_\mu u_g s_\nu^*)=f(d(\mu)-d(\nu))s_\mu u_g s_\nu^*
\qforal f\in \widehat{\fG(P)},\, g\in G,\, (\mu,\nu)\in \Lambda\tensor[_g]{\times}{_s}\Lambda.
\]
Denote by $\mathcal{O}_{G,\Lambda}^\gamma:=\{a \in \mathcal{O}_{G,\Lambda}:\gamma_z(a)=a \text{ for all }f\in \widehat{\fG(P)} \}$, the \emph{fixed point algebra} of $\gamma$ in $\O_{G,\Lambda}$. Then we obtain a faithful expectation
$E:\mathcal{O}_{G,\Lambda} \to \mathcal{O}_{G,\Lambda}^\gamma$ such that
\[
E(s_\mu u_g s_\nu^*)=\delta_{d(\mu),d(\nu)}s_\mu u_g s_\nu^*\qforal  g\in G\text{ and } (\mu,\nu)\in \Lambda\tensor[_g]{\times}{_s}\Lambda.
\]

\subsubsection{Connections with \cite{ EPS18, LY17_2}}
\label{SSS:connections}

If $P=\bN^k$ and $\Lambda^0$ is finite, then $\mathcal{O}_{G,\Lambda}=\ca_u(G,\Lambda)$ and coincides with the one introduced in \cite{LY17_2}.

In the case of self-similar directed graphs $(G,E)$, the C*-algebra $\O_{G,E}$ is isomorphic to the one defined in \cite[Definition 2.2]{EPS18}. In fact, to distinguish, let us temporarily use $\tilde \O_{G,E}$ to denote the C*-algebra defined in \cite[Definition 2.2]{EPS18}, which is generated by
$\{\tilde p_v: v\in E^0\}\cup \{ \tilde s_a: a\in E^1\}\cup \{\tilde u_{g,v}:g\in G, v\in E^0\}$.
 On one hand, the universal property of $\tilde\O_{G,E}$, yields a homomorphism $\pi_1:\tilde\O_{G,E}\to \O_{G,E}$ such that
$\pi_1(\tilde p_v)= s_v$, $\pi_1(\tilde s_a)= s_a$, and $\pi_1(\tilde u_{g,v})=u_gs_v$.
On the other hand, one can easily check that $\{\tilde s_\mu, \sum_{v\in E^0} \tilde u_{g,v}: \mu\in E^*, g\in G\}$ is a self-similar $(G, E)$-family in $M(\tilde \O_{G,E})$ (also cf.~\cite[2.1]{EPS18}). By the universal property of $\O_{G,E}$, there is a homomorphism $\pi_2$ from $\O_{G,E}$ to $\tilde\O_{G,E}$ such that
$\pi_2(s_\mu u_g s_v^*)=\tilde s_\mu \tilde u_{g, s(\nu)} \tilde s_\nu^*$.
It is easy to see that $\pi_1$ and $\pi_2$ are the inverse of each other. Thus
$\tilde\O_{G,E}$ and $\O_{G,E}$ are isomorphic.

With that said, one could also define $\O_{G,\Lambda}$ similar to \cite[Definition~2.2]{EPS18} by giving generators and relations directly. There are two main reasons why we choose $\ca_u(G,\Lambda)$ as a bridge to define $\O_{G,\Lambda}$: (a) keep the formulation as neat as possible; (b) notice that all relations of elements in $\O_{G,\Lambda}$ can be easily derived from $\ca_u(G,\Lambda)$.

%%%%%%%%%%%

\subsection{A product system associated to $(G, \Lambda)$}
Let $(G,\Lambda)$ be a self-similar $P$-graph.
Completely similar to \cite[Section 4]{LY17_2}, we associate a product system to a self-similar $P$-graph $(G,\Lambda)$. In what follows, we only sketch the main ideas behind.

Restricting the action of $G$ to $\Lambda^0$, we obtain a full crossed product C*-algebra $A_{G,\Lambda}:=C_0(\Lambda^0) \rtimes G=\overline{\mathrm{span}}\{i(\delta_v)i(g):v \in \Lambda^0,g \in G\}$ satisfying $i(g)i(\delta_v)=i(\delta_{g \cdot v})i(g)$ for all $v \in \Lambda^0,g \in G$.
Then for each $\mu \in \Lambda$, define a closed $A_{G,\Lambda}$-submodule of $A_{G,\Lambda}$ by
\[
X_{G,\Lambda,\mu}:=i(\delta_{s(\mu)})A_{G,\Lambda}=\overline{\mathrm{span}}\{i(\delta_{s(\mu)})i(g):g \in G\}.
\]
For $p \in P$, define a right Hilbert $A_{G,\Lambda}$-module $X_{G,\Lambda,p}$ as follows:
\begin{align*}
X_{G,\Lambda,p} := \begin{cases}
    {\displaystyle \bigoplus_{\mu \in \Lambda^p}}X_{G,\Lambda,\mu} &\text{ if $p \neq 0_P$} \\
    A_{G,\Lambda} &\text{ if $p=0_P$}.
\end{cases}
\end{align*}
Let $X_{G,\Lambda}:=\amalg_{p \in P}X_{G,\Lambda,p}$.
For $0_P \neq p \in P$ and $\mu \in \Lambda^p$, let $\chi_\mu \in X_{G,\Lambda,p}$ be defined as
\begin{align*}
\chi_\mu(\nu):= \begin{cases}
    i(\delta_{s(\mu)})  &\text{ if $\mu=\nu$} \\
    0 &\text{ otherwise }.
\end{cases}
\end{align*}
%%%%%%%
For $p \in P \setminus \{0_P\}$, there exists a nondegenerate covariant homomorphism $(\pi_{G,\Lambda,p},U_{G,\Lambda,p})$ for $A_{G,\Lambda}$ in $\mathcal{L}(X_{G,\Lambda,p})$ such that $\pi_{G,\Lambda,p}(\delta_v)$ is the projection from $X_{G,\Lambda,p}$ onto $\oplus_{\mu \in v \Lambda^p}X_{G,\Lambda,\mu}$, and
\[
U_{G,\Lambda,p}(g)(\chi_\mu a)=\chi_{g \cdot \mu}i(g \vert_\mu)a \qforal v \in \Lambda^0,\, g \in G,\, \mu \in \Lambda^p,\, a \in X_{G,\Lambda,\mu}.
\]
By \cite[Theorem~2.61]{Wil07}, there exists a nondegenerate homomorphism $\phi_{G,\Lambda,p}:A_{G,\Lambda} \to \mathcal{L}(X_{G,\Lambda,p})$ such that $\phi_{G,\Lambda,p}(i(\delta_v)i(g))=\pi_{G,\Lambda,p}(\delta_v)U_{G,\Lambda,p}(g)$. Since $\mathcal{L}(\oplus_{\mu \in v \Lambda^p}X_{G,\Lambda,\mu})=\mathcal{K}(\oplus_{\mu \in v \Lambda^p}X_{G,\Lambda,\mu})$ for all $v \in \Lambda^0$,  one has $\phi_{G,\Lambda,p}(A_{G,\Lambda}) \subseteq \mathcal{K}(X_{G,\Lambda,p})$ (cf.~\cite{LY17_2}).
%%%%%%%

Then one can define a product on $X_{G,\Lambda}$ as follows: for $p$, $q \in P \setminus \{0_P\}, \mu \in \Lambda^p$, $\nu \in \Lambda^q$, $g$, $h \in G$,
\begin{align*}
(\chi_\mu j(g)) \cdot (\chi_\nu j(h))
 := \begin{cases}
    \chi_{\mu (g \cdot \nu)} j(g|_\nu h) &\text{ if } s(\mu)=r(g \cdot \nu), \\
    0 &\text{ otherwise}.
\end{cases}
\end{align*}
Then with the above product $X_{G,\Lambda}$ is a product system over $P$ with coefficient $A_{G,\Lambda}$.
%%%%%
As in \cite{LY17_2}, one also has $\varphi_{G,\Lambda,p}(A_{G,\Lambda}) \subseteq \mathcal{K}(X_{G,\Lambda,p})$ for every $p\in P$.

\begin{prop}\label{P:OO}
Let $(G,\Lambda)$ be a self-similar $P$-graph. Denote by $j_X$ the universal Cuntz-Pimsner covariant representation of $X_{G,\Lambda}$ in $ \mathcal{O}_{X_{G,\Lambda}}$. Then
\begin{enumeratei}
\item
$\mathcal{O}_{G,\Lambda} \cong \mathcal{O}_{X_{G,\Lambda}}$;

\item 
$\mathcal{O}_{G,\Lambda}$ is nuclear if $G$ is amenable.
\end{enumeratei}
\end{prop}

\begin{proof}
(i) Since $j_{X,0_P}$ is a nondegenerate homomorphism, there exists a unique extension $j_{X,0_P}:M(A_{G,\Lambda}) \to M(\mathcal{O}_{X_{G,\Lambda}})$.
Also, one can verify that $\{j_{X,0_P}(i(g)),j_{X,d(\mu)}(\chi_\mu): g\in G, \mu\in\Lambda\}$ is a self-similar $(G,\Lambda)$-family in $M(\O_{X_{G,\Lambda}})$.
Therefore there is a homomorphism $\varPi:\O_{G,\Lambda}\to \mathcal{O}_{X_{G,\Lambda}}$ satisfying
%%%%%%%%%
\[
\varPi(s_\mu u_g s_\nu^*)=j_{X,d(\mu)}(\chi_\mu )j_{X,0_P}(i(g))j_{X,d(\nu)}(\chi_\nu)\qforal g\in G,\, (\mu,\nu)\in \Lambda\tensor[_g]{\times}{_s}\Lambda.
\]

In what follows, we construct the inverse of $\varPi$. To do so, we first define two homomorphisms $h:C_0(\Lambda^0) \to \ca_u(G,\Lambda)$ and $V:G \to \ca_u(G,\Lambda)$ by $h(\delta_v):=s_v$ for all $v \in \Lambda^0$ and $V_g:=u_g$ for all $g \in G$, respectively. Then $(h,V)$ is a nondegenerate covariant homomorphism for $A_{G,\Lambda}$ in $\ca_u(G,\Lambda)$. By \cite[Theorem~2.61]{Wil07}, there exists a homomorphism $\psi_0:A_{G,\Lambda} \to \mathcal{O}_{G,\Lambda}$ such that $\psi_0(i(\delta_v)i(g))=s_v u_g$ for all $v \in \Lambda^0$ and $g \in G$.

For $p \in P \setminus \{0_P\}$, define a map $\psi_p:X_{G,\Lambda,p} \to \mathcal{O}_{G,\Lambda}$ by
\[
\psi_p\left(\sum_{\mu \in \Lambda^p} \chi_\mu a_\mu\right):=\sum_{\mu\in\Lambda^p} s_\mu \psi_0(a_\mu).
\]
By piecing $\{\psi_p\}_{p \in P}$ together we obtain a Toeplitz representation $\psi:X_{G,\Lambda} \to \mathcal{O}_{G,\Lambda}$. To see that $\psi$ is Cuntz-Pimsner covariant, fix $p \in  P \setminus \{0_P\},v \in \Lambda^0,g \in G$, we compute that
\begin{align*}
\psi_p^{(1)}(\phi_{G,\Lambda,p}(i(\delta_v)i(g)))&=\psi_p^{(1)}(\pi_{G,\Lambda,p}(\delta_v) U_{G,\Lambda,p}(g))
\\&=\sum_{\mu \in v \Lambda^p}\psi_p^{(1)}(\Theta_{\chi_\mu,\chi_\mu}U_{G,\Lambda,p}(g))
\\&=\sum_{\mu \in v \Lambda^p} \psi_p(\chi_\mu) \psi_p(\chi_{g^{-1} \cdot \mu}i_G(g^{-1}\vert_\mu))^*
\\&=\sum_{\mu \in v \Lambda^p} s_\mu (s_{g^{-1} \cdot \mu} u_{g^{-1}\vert_\mu})^*
\\&=\sum_{\mu \in v \Lambda^p} s_\mu(u_{g^{-1}}s_\mu)^*
\\&=s_v u_g
\\&=\psi_0(i(\delta_v)i(g)).
\end{align*}
So there exists a homomorphism $\rho: \mathcal{O}_{X_{G,\Lambda}} \to \mathcal{O}_{G,\Lambda}$ such that $\rho \circ j_X=\psi$.

It is straightforward to check that $\rho \circ \varPi=\id_{\mathcal{O}_{G,\Lambda}}$ and $\varPi \circ \rho=\id_{\mathcal{O}_{X_{G,\Lambda}}}$. So $\pi$ is an isomorphism.

(ii) Since $G$ is amenable, $A_{G,\Lambda}$ is nuclear. Then (ii) immediately follows from (i) above and \cite[Theorem~3.21]{AM15}.
%%%%%%%%%%%%
\end{proof}

\subsection{A groupoid associated to $(G, \Lambda)$}
The construction in this subsection is similar to \cite[Section 5]{LY17_2}.
Let $(G,\Lambda)$ be a self-similar $P$-graph. Denote by $C(P,G)$ the set of all mappings from $P$ to $G$, which is a group under the pointwise multiplication. For $h \in \fG(P)$ and $f \in C(P,G)$, define $\T_{h}(f) \in C(P,G)$ by
\begin{align*}
\T_{h}(f)(p)&:= \begin{cases}
    f(p-h) &\text{if $p-h \in P$} \\
    1_G &\text{otherwise}.
\end{cases}
\end{align*}
For $f_1,f_2 \in C(P,G)$, define an equivalence relation $f_1 \sim f_2$ if there exists $p \in P$ such that $f_1(q)=f_2(q)$ for all $q \in P$ with $q- p \in P$. Let $Q(P,G)$ be the quotient group $C(P,G)/\!\!\sim$.  For $f\in C(P,G)$, we write $[f] \in Q(P,G)$. Then $\T_h$ yields an automorphism, still denoted by $\T_h$, on $Q(P,G)$. Since $\T:\fG(P) \to \Aut( Q(P,G)),\ h\mapsto \T_h$, is a homomorphism, we obtain the semidirect product $Q(P,G) \rtimes_\T \fG(P)$.

For $g \in G$ and $x \in \Lambda^\infty$, define $g\cdot x\in \Lambda^\infty$ and $g \vert_x \in C(P,G)$ by
\begin{align*}
(g \cdot x)(p,q)&:=g \vert_{x(0_P,p)} \cdot x(p,q) \qforal  p,q \in P \text{ with } q-p \in P,\\
g \vert_x(p)&:=g \vert_{x(0_P,p)}\qforal p \in P.
\end{align*}

Define
\[
\mathcal{G}_{G,\Lambda}:=\left\{\big(\mu (g \cdot x);\T_{d(\mu)}([g \vert_{x}]),d(\mu)-d(\nu);\nu x\big):
\begin{matrix}
g \in G, (\mu,\nu)\in\Lambda\tensor[_g]{\times}{_s}\Lambda\\
x \in s(\nu)\Lambda^\infty
\end{matrix}
\right\},
\]
which is a subgroupoid of $\Lambda^\infty \times (Q(P,G) \rtimes_\T \fG(P) ) \times \Lambda^\infty$.

For $g \in G$ and $(\mu,\nu)\in\Lambda\tensor[_g]{\times}{_s}\Lambda$, define
\[
Z(\mu,g,\nu):=\big\{(\mu (g \cdot x);\T_{d(\mu)}([g \vert_{x}]),d(\mu)-d(\nu);\nu x):x \in s(\nu)\Lambda^\infty\big\}.
\]
Endow $\mathcal{G}_{G,\Lambda}$ with the topology generated by the basic open sets
\[
\mathcal{B}_{G,\Lambda}:=\big\{Z(\mu,g,\nu):g \in G, (\mu,\nu)\in\Lambda\tensor[_g]{\times}{_s}\Lambda\big\}.
\]
A similar argument of the proof of \cite[Theorem~5.8]{LY17_2} proves

\begin{prop}
Let $(G,\Lambda)$ be a self-similar $P$-graph. Then $\G_{G,\Lambda}$ is an ample groupoid with each $Z(\mu, g, \nu)$ being a compact open bisection.
\end{prop}

\begin{lem}\label{L:G-ape}
Let $(G,\Lambda)$ be a self-similar $P$-graph. Then
$(G,\Lambda)$ is aperiodic if and only if $\mathcal{G}_{G,\Lambda}$ is topologically principal.
\end{lem}

\begin{proof}
The proof is similar to that of \cite[Proposition 6.5]{LY17_2}.
First, suppose that $\mathcal{G}_{G,\Lambda}$ is not topologically principal. Then there exists $\mu \in \Lambda$ such that for any $x \in s(\mu)\Lambda^\infty$, we have $(\mathcal{G}_{G,\Lambda})_{\mu x}^{\mu x} \neq \{\mu x\}$. Fix $x \in s(\mu)\Lambda^\infty$. Then there exists $\big(\alpha (g \cdot y);\T_{d(\alpha)}([g \vert_y]),d(\alpha)-d(\beta);\beta y\big) \in \mathcal{G}_{G,\Lambda}$ such that $\mu x=\alpha (g \cdot y)=\beta y$ and $(\T_{d(\alpha)}([g \vert_y]),d(\alpha)-d(\beta)) \neq (1_{Q(P,G)},0_P)$. If $\T_{d(\alpha)}([g \vert_y]) \neq 1_{Q(P,G)}$, then there exists $p \in P$ such that $g \vert_{y(0_P,d(\mu)+p)} \neq 1_G$ and so that $\sigma^{d(\alpha)+p}(x)=g \vert_{y(0,d(\mu)+p)} \cdot \sigma^{d(\beta)+p}(x)$. If $d(\alpha) \neq d(\beta)$, then $\sigma^{d(\alpha)}(x)=g \vert_{y(0,d(\mu))} \cdot \sigma^{d(\beta)}(x)$. Hence $(G,\Lambda)$ is not $G$-aperiodic.

Conversely, suppose that $(G,\Lambda)$ is not $G$-aperiodic. Then there exists $v \in \Lambda^0$ satisfying that for any $x \in v \Lambda^\infty$ there exist $g \in G, p,q \in P$ with $g \neq 1_G$ or $p \neq q$, such that $\sigma^{p}(x)=g \cdot \sigma^q(x)$. So for any $x \in v\Lambda^\infty,(x;\T_p([g \vert_{\sigma^q(x)}]),p-q;x) \in (\mathcal{G}_{G,\Lambda})_{ x}^{ x} $. Hence $(\mathcal{G}_{G,\Lambda})_{ x}^{ x} \neq \{x\}$. Therefore $\mathcal{G}_{G,\Lambda}$ is not topologically principal.
\end{proof}

%%%%%%%%%%%%%%%%%
For later use, we record the following lemma whose proof is straightforward. 

\begin{lem}\label{operators in multiplier}
Let $A$ be a C*-algebra, $S$ be a set of generators of $A$, and $(T_i)_{i \in I}$ be a net of operators in $M(A)$ such that $(T_i)_{i \in I}$ is uniformly bounded. Suppose that for any $a \in S \cup S^*$, the nets $(T_ia)_{i \in I}$ and $(T_i^*a)_{i \in I}$ converge. Then $(T_i)_{i \in I},(T_i^*)_{i \in I}$ converge strictly and $(\lim_{i \in I}T_i)^*=\lim_{i \in I}T_i^*$.
\end{lem}

\begin{prop}\label{P:OG}
Let $(G,\Lambda)$ be a self-similar $P$-graph.
Then there exists a natural surjective homomorphism from $\mathcal{O}_{G,\Lambda}$ onto
$\ca(\mathcal{G}_{G,\Lambda})$.
\end{prop}

\begin{proof}
For any $\mu \in \Lambda$ and $g \in G$, let $S_\mu:=\mathds{1}_{Z(\mu,1_G,s(\mu))}$ and $U_g:=\sum_{v \in \Lambda^0}\mathds{1}_{Z(v,g,g^{-1} \cdot v)}$. By Lemma~\ref{operators in multiplier}, one has $U_g \in UM(\ca(\mathcal{G}_{G,\Lambda}))$.
Then $\{U, S\}$ is a self-similar $(G,\Lambda)$-family in $M(\ca(\mathcal{G}_{G,\Lambda}))$. So, by the universal property of $\O_{G,\Lambda}$, there exists a homomorphism $h_2:\O_{G,\Lambda} \to M(\ca(\mathcal{G}_{G,\Lambda}))$ such that $h_2(s_\mu u_g s_\nu^*)=S_\mu U_g S_\nu^*$ for all $g\in G$ and $(\mu,\nu)\in \Lambda\tensor[_g]{\times}{_s}\Lambda$.
As the proof of \cite[Theorem~5.9]{LY17_2}, one can check that the range of $h_2$ is $\ca(\G_{G,\Lambda})$.
\end{proof}

%%%%%%%%%%
\subsection{A Cuntz-Krieger uniqueness theorem}

In this subsection, we prove a Cuntz-Krieger uniqueness theorem for self-similar $P$-graph C*-algebras, which will be essentially used in Section \ref{S:primitive} to describe the structure for the primitive ideas of self-similar $k$-graph C*-algebras.

Let $(G,\Lambda)$ be a self-similar $P$-graph. To obtain our Cuntz-Krieger uniqueness theorem for $\O_{G,\Lambda}$, for some technical reasons, we have to assume that
\begin{align*}
\textsf{$g\cdot v=v$ \text{ for all } $g\in G$ \text{ and }$v\in \Lambda^0$.}
\tag{FV}
\end{align*}
Notice that, under Property~(FV), we have $\Lambda\tensor[_g]{\times}{_s}\Lambda = \{(\mu, \nu)\in \Lambda\times \Lambda: s(\mu)=s(\nu)\}=:\Lambda\times_s\Lambda$ for every $g\in G$.

\begin{lem}\label{Fourier}
Let $H$ be a countable discrete abelian group. Then $\int_{f \in \widehat{H}}f(h_0)\,df=0$ for all $0_H \neq h_0 \in H$.
\end{lem}

\begin{proof}
Define an $L^1(H)$ function $F$ by $F(0_H):=1$ and $F(h):=0$ for all $h \neq 0_H$. Then the Fourier transformation of $F$, denoted by $\widehat{F}$, belongs to $L^1(\widehat{H})$. 
By the Fourier inversion theorem (see \cite[Theorem~4.32]{Fol95}), we have
\[
0=F(h_0)=\int_{f \in \widehat{H}}f(h_0)\widehat{F}(f)\,df=\int_{f \in \widehat{H}}f(h_0)\,df.
\]
So we are done.
\end{proof}

%%%%%%%%%
Let $(G,\Lambda)$ be a self-similar $P$-graph. Let $\pi:\mathcal{O}_{G,\Lambda} \to B$ be a surjective homomorphism. Suppose that there exists a strongly continuous homomorphism $\alpha:\widehat{\fG(P)} \to \Aut(B)$ such that
$\pi$ intertwines $\alpha$ and $\gamma$.
Denote by
$B^\alpha:=\{a \in B:\alpha_f(a)=a \text{ for all }f \in \widehat{\fG(P)}  \}$ the \emph{fixed point algebra of $\alpha$}. By Lemma~\ref{Fourier}, we obtain a faithful expectation $F:B \to B^\alpha$ defined by
\[
F(x)=\int_{f\in \widehat{\fG(P)}}\alpha_f(x)\, df.\]
Clearly, one has
$F(\pi(s_\mu u_g s_\nu^*))=\delta_{d(\mu),d(\nu)}\pi(s_\mu u_g s_\nu^*)$ for all $g\in G$ and $(\mu, \nu) \in \Lambda\times_s \Lambda$.

The following can be regarded as a gauge-invariant uniqueness theorem for self-similar $P$-graph C*-algebras.

\begin{thm}\label{T:gissP}
Let $(G,\Lambda)$ be a self-similar $P$-graph with Property (FV), and $\pi:\mathcal{O}_{G,\Lambda} \to B$ be a surjective homomorphism. Suppose that
\begin{enumeratei}
\item there exists a strongly continuous homomorphism $\alpha:\widehat{\fG(P)} \to \Aut(B)$ such that $\pi$ intertwines $\alpha$ and the gauge action $\gamma$ of $\O_{G,\Lambda}$;
\item $\pi(s_v) \neq 0$ for all $v \in \Lambda^0$;
\item
there exists an expectation $\Phi:B \to \pi(\mathcal{D}_\Lambda)$ such that
\[
\Phi(\pi(s_\mu u_g s_\nu^*))=\delta_{\mu,\nu}\delta_{g,1_G}\pi(s_\mu s_\mu^*) \qforal g\in G, (\mu, \nu) \in \Lambda\times_s\Lambda.
\]
\end{enumeratei}
Then $\pi$ is injective.
\end{thm}

\begin{proof}
Let $F:B \to B^\alpha$ be the expectation mentioned above, and $E$ be the faithful expectation obtained in Section~\ref{SSS:gauge}.
It is obvious that $\pi$ intertwines the two expectations $E$ and $F$: $F \circ \pi=\pi \circ E$. So, in order to prove that $\pi$ is injective, it is enough to show that $\pi$ is injective on $\mathcal{O}_{G,\Lambda}^\gamma$ by \cite[Proposition~3.11]{Kat03}.

For $p \in P$, define $A_p:=\overline{\mathrm{span}}\{s_\mu u_g s_\nu^*:\mu,\nu \in\Lambda^p, s(\mu)=s(\nu)\}$. Then $A_p$ is a C*-subalgebra of $\mathcal{O}_{G,\Lambda}^\gamma$ for all $p \in P$; $A_p \subseteq A_q$ for all $p, q \in P$ with $q-p \in P$; and $\mathcal{O}_{G,\Lambda}^\gamma=\overline{\bigcup_{p \in P}A_p}$. So in order to show that $\pi$ is injective on $\mathcal{O}_{G,\Lambda}^\gamma$, it suffices to prove that $\pi$ is injective on $A_p$ for all $p \in P$.

For $p \in P, v \in \Lambda^0$, define $A_{p,v}:=\overline{\mathrm{span}}\{s_\mu u_g s_\nu^*\in A_p:\mu,\nu \in\Lambda^p v\}$. Then $A_{p,v}$ is a C*-subalgebra of $A_p$ for all $p \in P,v \in \Lambda^0$, and $A_p \cong \oplus_{v \in \Lambda^0}A_{p,v}$ for all $p \in P$. Hence to show that $\pi$ is injective on $A_p$ for all $p \in P$, we only need to prove that $\pi$ is injective on $A_{p,v}$ for all $p \in P$ and $v \in \Lambda^0$.

Fix $p \in P$ and $v \in \Lambda^0$. Denote by $\{e_{\mu,\nu}\}_{\mu,\nu \in \Lambda^p v}$ the standard generators of $K(\ell^2(\Lambda^p v))$. Let $\{i(g)\}_{g \in G}$ be the generating unitaries of $\ca(G)$. There exists a homomorphism $\rho_1:K(\ell^2(\Lambda^p v)) \to A_{p,v}$ such that $\rho_1(e_{\mu,\nu})=s_\mu s_\nu^*$ for all $\mu,\, \nu \in \Lambda^pv$. It follows from Lemma~\ref{operators in multiplier} that
 $\sum_{\mu \in \Lambda^p v}s_\mu u_g s_\mu^* \in UM(A_{p,v})$ for every $g\in G$. So there exists a homomorphism $\rho_2:\ca(G) \to UM(A_{p,v})$ such that $\rho_2(i(g))=\sum_{\mu \in \Lambda^pv}s_\mu u_g s_\mu^*$ for all $g \in G$. One can easily verify that the images of $\rho_1$ and $\rho_2$ commute. By \cite[Theorem~6.3.7]{Mur90} there exists a surjective homomorphism $\rho:K(\ell^2(\Lambda^pv)) \otimes_{\text{max}} \ca(G) \to A_{p,v}$ such that $\rho(e_{\mu,\nu} \otimes i(g))=s_\mu u_g s_\nu^*$ for all $\mu,\, \nu \in \Lambda^pv,\, g \in G$. By Condition~(iii), there exists an expectation $\Phi:\pi(A_{p,v}) \to \ca(\{\pi(s_\mu s_\mu^*):\mu \in \Lambda^p v\})$ such that $\Phi(\pi(s_\mu u_g s_\nu^*))=\delta_{\mu,\nu}\delta_{g,1_G}\pi(s_\mu s_\mu^*)$ for all $\mu,\nu \in \Lambda^p v,g \in G$. By \cite[Proposition~4.1.9]{BO08}, there exists a faithful expectation $\Psi:K(\ell^2(\Lambda^pv)) \otimes_{\text{max}} \ca(G) \to \ca(\{e_{\mu,\mu}:\mu \in \Lambda^p v\})$ such that $\Psi(e_{\mu,\nu} \otimes i(g))=\delta_{\mu,\nu}\delta_{g,1_G}e_{\mu,\mu}$ for all $\mu,\,\nu \in \Lambda^pv,\, g \in G$. It is straightforward to see that $\pi \circ \rho\circ \Psi=\Phi \circ \pi \circ \rho$. Since $\pi(s_v) \neq 0$, one can check that $\pi \circ \rho$ is injective on $\ca(\{s_\mu s_\mu^*:\mu \in \Lambda^p v\})$. So by \cite[Proposition~3.11]{Kat03}, $\pi \circ \rho$ is injective on $A_{p,v}$.
Therefore, $\pi$ is injective.
\end{proof}

\begin{prop}\label{P:OCG}
%Let $P$ be a cancellative abelian semigroup,
Let $(G,\Lambda)$ be a self-similar $P$-graph with Property (FV). Then $\O_{G,\Lambda}\cong \ca(\G_{G,\Lambda})$.
\end{prop}

\begin{proof}
By Proposition~\ref{P:OG}, there exists a surjective homomorphism $\pi:\mathcal{O}_{G,\Lambda} \to \ca(\mathcal{G}_{G,\Lambda})$ such that $\pi(s_\mu u_g s_\nu^*)=\mathds{1}_{Z(\mu,g,\nu)}$ for all $g \in G$ and $(\mu,\nu)\in\Lambda\tensor[_g]{\times}{_s}\Lambda$. Let $c:\mathcal{G}_{G,\Lambda} \to \fG(P)$ be the continuous $1$-cocycle defined by
\[
c(\mu (g \cdot x);\T_{d(\mu)}([g \vert_{x}]),d(\mu)-d(\nu);\nu x):=d(\mu)-d(\nu).
\]
By \cite[Proposition~II.5.1]{Ren80}, there exists a strongly continuous homomorphism $\alpha: \widehat{\fG(P)}  \to \Aut(\ca(\mathcal{G}_{G,\Lambda}))$ such that
$\alpha\circ \pi=\pi\circ \gamma$.
It is straightforward to see that $\pi(s_v) \neq 0$ for all $v \in \Lambda^0$. By \cite[Proposition~II.4.8]{Ren80} there exists an expectation $\E:\ca(\mathcal{G}_{G,\Lambda}) \to C_0(\mathcal{G}_{G,\Lambda}^{0})$ such that $\E(f)=f \vert_{\mathcal{G}_{G,\Lambda}^{0}}$ for all $f \in C_c(\mathcal{G}_{G,\Lambda})$. The Stone-Weierstrass theorem implies that $\pi(\mathcal{D}_\Lambda)=C_0(\mathcal{G}_{G,\Lambda}^{0})$. Since $(G,\Lambda)$ is pseudo free, we have $\E(\pi(s_\mu u_g s_\nu^*))=\delta_{\mu,\nu}\delta_{g, 1_G}\pi(s_\mu s_\mu^*)$ for all $g\in G$ and $(\mu,\nu)\in\Lambda\times_s\Lambda$. By Theorem~\ref{T:gissP}, $\pi$ is injective.
\end{proof}

\begin{cor}
\label{C:OOCG}
Let $(G,\Lambda)$ be a self-similar $P$-graph with Property (FV). Then 
\begin{enumeratei}
\item
$\O_{G,\Lambda}\cong\O_{X_{G,\Lambda}}\cong \ca(\G_{G,\Lambda})$;
\item 
$\ca(\G_{G,\Lambda})\cong\ca_r(\G_{G,\Lambda})$, $\O_{G,\Lambda}$ satisfies the UCT, and $\G_{G,\Lambda}$ is amenable, provided $G$ is amenable.
\end{enumeratei}
\end{cor}

\begin{proof}
Combining Propositions \ref{P:OO} and \ref{P:OCG} yields (i). 
The proof of (ii) is similar to \cite[Theorem~6.6(i) and Proposition~6.7]{LY17_2}.
\end{proof}

We are now ready to state and prove the main result of this subsection. It is a Cuntz-Krieger uniqueness theorem for self-similar $P$-graphs C*-algebras, which generalizes the one for $P$-graphs C*-algebras in \cite[Corollary~2.8]{CKSS14}, and plays a key role in Section~\ref{S:primitive}.

\begin{thm}\label{T:aCKU}
Let $B$ be a C*-algebra, $(G,\Lambda)$ be an aperiodic self-similar $P$-graph with Property (FV), and $\pi:\mathcal{O}_{G,\Lambda} \to B$ be a homomorphism such that $\pi(s_v) \neq 0$ for all $v \in \Lambda^0$. Then $\pi$ is injective.
\end{thm}

\begin{proof}
Since $(G,\Lambda)$ is $G$-aperiodic, $\mathcal{G}_{G,\Lambda}$ is topologically principal due to Lemma \ref{L:G-ape}. It follows from Propositions~\ref{P:OO}, \ref{P:OCG} and \cite[Corollary~5.6.17, Theorem~5.6.18]{BO08} that $\mathcal{O}_{G,\Lambda} \cong \ca_r(\mathcal{G}_{G,\Lambda})$. Thus we may regard $\pi$ a homomorphism from $\ca_r(\mathcal{G}_{G,\Lambda})$ to $B$. Since $\pi(s_v) \neq 0$ for all $v \in \Lambda^0$, by Theorem~\ref{CK thm gpoid} $\pi$ is injective.
\end{proof}

%%%%%%%%%%%%%%
%%%%%%%%%%
\subsection{Properties of self-similar $k$-graphs C*-algebras -- from unital to nonunital}

The main purpose of this subsection is to study the properties of (not necessarily unital) self-similar $k$-graph C*-algebras.

The theorem below can be proved completely similar to \cite[Theorem~5.10]{LY17_2}. 

\begin{thm}\label{T:OOG}
Let $(G,\Lambda)$ be a self-similar $k$-graph with $G$ amenable.
Then $\mathcal{O}_{G,\Lambda} \cong \mathcal{O}_{X_{G,\Lambda}}  \cong \ca(\mathcal{G}_{G,\Lambda}) \cong \ca_r(\G_{G,\Lambda})$.
\end{thm}

In the following gauge-invariant uniqueness type result, we identify $\O_{G,\Lambda}$ with $\O_{X_{G,\Lambda}}$ via $\varPi$ used in the proof of Proposition \ref{P:OO}. 

\begin{prop}\label{P:GUIT}
Let $(G,\Lambda)$ be a self-similar $k$-graph, let $B$ be a C*-algebra, and let $\pi:\mathcal{O}_{G,\Lambda} \to B$ be a surjective homomorphism. Suppose that
\begin{enumeratei}
\item there exists a strongly continuous homomorphism $\alpha:\bT^k \to \Aut(B)$ such that $\alpha_z(\pi(s_\mu u_g s_\nu^*))=z^{d(\mu)-d(\nu)}\pi(s_\mu u_g s_\nu^*)$ for all $z \in \bT^k$, $g\in G$, $(\mu,\nu)\in\Lambda\tensor[_g]{\times}{_s}\Lambda$;
\item $\pi$ is injective on $A_{G,\Lambda}$.
\end{enumeratei}
Then $\pi$ is injective. 
\end{prop}

\begin{proof}
It follows immediately from Theorem~\ref{T:GIUCLSV} and Proposition~\ref{P:OO}.
\end{proof}

Now we can prove the nonunital version of \cite[Theorems 6.6, 6.13]{LY17_2}.

\begin{thm}\label{T:nonunital}
Let $(G,\Lambda)$ be a self-similar $k$-graph with $G$ amenable. Then
\begin{enumeratei}
\item $\mathcal{O}_{G,\Lambda}$ is nuclear;
\item $\mathcal{O}_{G,\Lambda}$ satisfies the UCT;
\item $\mathcal{O}_{G,\Lambda}$ is simple if and only if $(G,\Lambda)$ is cofinal and aperiodic;
\item if $\mathcal{O}_{G,\Lambda}$ is simple, then $\mathcal{O}_{G,\Lambda}$ is stably finite if and only if $\Lambda$ has faithful graph traces;
\item if $\mathcal{O}_{G,\Lambda}$ is simple and $S(\mathcal{G}_{G,\Lambda})$ is almost unperforated, then $\mathcal{O}_{G,\Lambda}$ is purely infinite if and only if $\Lambda$ has no faithful graph traces.
\end{enumeratei}
\end{thm}

\begin{proof}
When $|\Lambda^0|<\infty$, the theorem is already showed in \cite[Theorems~6.6 and 6.13]{LY17_2}. So, in the sequel, we only need to prove it when $|\Lambda^0|=\infty$. 

(i) is from Proposition \ref{P:OO}; and (ii)-(iii) are proved similar to \cite[Theorem 6.6]{LY17_2}.

For (iv), assume that $\mathcal{O}_{G,\Lambda}$ is simple. Suppose that $\mathcal{O}_{G,\Lambda}$ is stably finite. By Theorem~\ref{T:RS}, there exists a faithful semifinite trace $\tau$ on $\mathcal{O}_{G,\Lambda}$ such that $0<\tau(s_\mu s_\mu^*)<\infty$ for all $\mu \in \Lambda$. Define $\ft:\Lambda^0 \to [0,\infty)$ by $\ft(v):=\tau(s_v)$ for all $v \in \Lambda^0$. So $\ft$ is a faithful graph trace.

Conversely, suppose that $\Lambda$ has a faithful graph trace $\ft$. Then there exists a linear functional $\phi:\mathrm{span}\{s_\mu s_\mu^*:\mu \in \Lambda\} \to \mathbb{C}$ satisfying $\phi(s_\mu s_\mu^*)=\ft(s(\mu))$ for all $\mu\in\Lambda$. Denote by $\mathcal{A}$ the $*$-subalgebra of $\mathcal{O}_{G,\Lambda}$ generated by $\{s_\mu u_g s_\nu^*:g\in G, (\mu,\nu)\in \Lambda\tensor[_g]{\times}{_s}\Lambda\}$. It follows from Theorem~\ref{T:OOG} that there exists a linear map $E:\mathcal{A} \to \mathrm{span}\{s_\mu s_\mu^*:\mu \in \Lambda\}$ such that $E(s_\mu u_g s_\nu^*)=\delta_{\mu,\nu}\delta_{g,1_G}s_\mu s_\mu^*$ for all $g\in G$ and $(\mu,\nu)\in \Lambda\tensor[_g]{\times}{_s} \Lambda$. So $\tau:=\phi \circ E$ is a linear functional on $\mathcal{A}$ such that $\tau(ab)=\tau(ba)$ for all $a,b \in \mathcal{A}$. Define $\langle\cdot,\cdot\rangle:\mathcal{A} \times \mathcal{A} \to \mathbb{C}$ by $\langle a,b\rangle:=\tau(b^*a)$. It is straightforward to check that $\langle\cdot,\cdot\rangle$ is an inner product (the faithfulness of $\ft$ guarantees the axiom $\langle a,a\rangle=0 \implies a=0$) and $\mathcal{A}$ becomes a Hilbert algebra. Denote by $\mathcal{H}$ the completion of $\mathcal{A}$ as an inner product space. For $\mu\in \Lambda,g \in G,a \in \mathcal{A}$, define $S_\mu a:=s_\mu a$ and $U_ga:=u_g a$.
This yields a representation $\pi: \O_{G,\Lambda}\to B(\H)$.
Since $\mathcal{O}_{G,\Lambda}$ is simple, $\mathcal{O}_{G,\Lambda}$ embeds in $\mathcal{A}''$. By Theorem~\ref{exist a semifinite trace}, there exists an ultraweakly lower semicontinuous trace $h$ on $\mathcal{A}''$ which is finite on $a^*a$ for all $a \in \mathcal{A}$. Then $h$ restricts to a densely defined, lower semicontinuous trace (hence semifinite due to Lemma~\ref{densely def lower semicont trace is semifin}) on $\mathcal{O}_{G,\Lambda}$. We deduce that $h$ is faithful because $\mathcal{O}_{G,\Lambda}$ is simple. Therefore $\mathcal{O}_{G,\Lambda}$ is stably finite by Theorem~\ref{T:RS}.

Finally, (v) follows from (iv) and Theorem~\ref{T:RS}.
\end{proof}

\section{Gauge- and Diagonal-Invariant Ideals of Self-Similar $k$-Graph C*-Algebras}
\label{S:gdinv}

In this section, we characterize both gauge- and diagonal-invariant ideals of self-similar $k$-graph C*-algebras (see Definition~\ref{D:hsgdi}). Our approach is inspired by \cite{RSY03}.

Let $(G,\Lambda)$ be a self-similar $k$-graph. By Theorem~\ref{T:OOG} there is a faithful expectation $\E_{G,\Lambda}:\mathcal{O}_{G,\Lambda} \to \mathcal{D}_{\Lambda}$  such that
\begin{align}
\label{E:EGLambda}
\E_{G,\Lambda}(s_\mu u_g s_\nu^*)=\delta_{\mu,\nu} \delta_{g,1_G} s_\mu s_\mu^*\qforal g\in G,\, (\mu,\nu)\in \Lambda\tensor[_g]{\times}{_s}\Lambda.
\end{align}
%\end{notation}

\begin{defn}\label{D:hsgdi}
Let $(G,\Lambda)$ be a self-similar $k$-graph.  A subset $H$ of $\Lambda^0$ is said to be
\begin{itemize}
\item \emph{$G$-hereditary} if $G \cdot H \subseteq H$ and $s(H \Lambda) \subseteq H$;
\item \emph{$G$-saturated} if $G \cdot H \subseteq H$ and $r(\Lambda H)\subseteq H$.
\end{itemize}
Let $\gamma$ be the gauge action of $\O_{G,\Lambda}$ (see Subsection~\ref{SSS:gauge}). An ideal $I$ of $\mathcal{O}_{G,\Lambda}$ is said to be
\begin{itemize}
\item \emph{gauge-invariant} if $\gamma_z(I) \subseteq I$ for all $z \in \bT^k$;
\item \emph{diagonal-invariant} if $\E_{G,\Lambda}(I) \subseteq I$.
\end{itemize}
\end{defn}

\medskip
%\begin{notation}
Let $(G,\Lambda)$ be a self-similar $k$-graph. For a given $G$-hereditary subset $H$ of $\Lambda^0$,  let $I(H)$ denote the ideal of $\mathcal{O}_{G,\Lambda}$ generated by $\{s_v:v \in H\}$. It is not hard to check that $I(H)=\overline{\mathrm{span}}\{s_\mu u_g s_\nu^*:s(\mu),s(\nu) \in H\}$, and that $I(H)$ is both gauge- and diagonal-invariant. Conversely, for a given gauge- and diagonal-invariant ideal $I$ of $\mathcal{O}_{G,\Lambda}$, let $H(I):=\{v \in \Lambda^0:s_v \in I\}$. Then $H(I)$ is $G$-hereditary and $G$-saturated.
%\end{notation}

The aim of this section is to show that there is a one-to-one correspondence between the set of all $G$-hereditary and $G$-saturated subsets of $\Lambda^0$ and the set of all gauge- and diagonal-invariant ideals of $\O_{G,\Lambda}$.
For this, we need some preparation.

\begin{lem}\label{L:emb}
Let $(G,\Lambda)$ be a self-similar $k$-graph and let $H$ be a $G$-hereditary subset of $\Lambda^0$. Then $(G,H\Lambda )$ is a pseudo free self-similar $k$-graph. Moreover, $\mathcal{O}_{G,H \Lambda }$ naturally embeds in $\mathcal{O}_{G,\Lambda}$, and $\mathcal{O}_{G,H \Lambda }$ is a full corner of $I(H)$.
\end{lem}

\begin{proof}
The first part is straightforward to verify. We prove the ``Moreover" part only.
Let  $\{u,s\}$ and $\{w, t\}$ be the universal self-similar $(G,\Lambda)$-family and $(G,H\Lambda)$-family, respectively.
Notice that restricted to $H\Lambda$, $\{u,s\}$ is also a self-similar $(G,H\Lambda)$-family (cf.~Remark~\ref{R:ssfamily} (ii)). So by the universal property of $\O_{G,H\Lambda}$, there is a homomorphism $\pi: \O_{G,H\Lambda}\to \O_{G,\Lambda}$ such that $\pi(w_g)=u_g$ and $\pi(t_\mu)=s_\mu$ for all $g\in G$ and $\mu\in H\Lambda$.
The universal property of $\O_{G,H \Lambda}$ yields a strongly continuous homomorphism $\alpha:\bT^k \to \Aut(\pi(\mathcal{O}_{G,H \Lambda }))$ such that $\alpha_z(\pi(t_\mu w_g t_\nu^*))=z^{d(\mu)-d(\nu)}\pi(t_\mu w_g t_\nu^*)$ for all $z \in \bT^k,\ \mu,g \in G$, $(\mu, \nu)\in H\Lambda\tensor[_g]{\times}{_s}H\Lambda$. By \cite[Proposition~4.1.9]{BO08}, there exist faithful expectations $E_1:C_0(H)\rtimes G \to C_0(H)$ and $E_2:C_0(\Lambda^0) \rtimes G \to C_0(\Lambda^0)$ such that $E_1(t_v w_g)=\delta_{g,1_G}t_v,E_2(s_{v'}u_{g'})=\delta_{g',1_G}s_{v'}$ for all $v \in H,v' \in \Lambda^0,g,g' \in G$. Since $\pi \circ E_1=E_2 \circ \pi$ and $\pi$ is injective on $C_0(H)$, we deduce that $\pi$ is injective on $C_0(H) \rtimes G$. By Proposition~\ref{P:GUIT}, $\pi$ is injective.

By Lemma~\ref{operators in multiplier}, $p:=\sum_{v \in H}s_v$ is a projection of $M(I(H))$. It is straightforward to check that $pI(H)p=\mathcal{O}_{G,H \Lambda }$ and $\overline{\mathrm{span}}I(H)pI(H)=I(H)$. Hence $\mathcal{O}_{G,H \Lambda }$ is a full corner of $I(H)$.
\end{proof}

\begin{lem}\label{L:HIH}
Let $(G,\Lambda)$ be a self-similar $k$-graph, $H$ be a $G$-hereditary and $G$-saturated subset of $\Lambda^0$, and $v \in \Lambda^0$. Then $v \in H$ if and only if $s_v \in I(H)$.
\end{lem}

\begin{proof}
It is straightforward to see that if $v \in H$ then $s_v \in I(H)$. Conversely, suppose that $s_v \in I(H)$. Since $H$ is $G$-hereditary, we can approximate $s_v$ by $\spn\{s_\mu u_g s_\nu^*:s(\mu)=g \cdot s(\nu) \in H\}$.
Since $\E_{G,\Lambda}$ is an expectation from $\mathcal{O}_{G,\Lambda}$ to $\mathcal{D}_\Lambda$, we can furthermore approximate $s_v$ by $\spn\{s_\mu s_\mu^*:s(\mu) \in H\}$.
Let $\epsilon>0$ be small enough. Then there are $p\in\bN^k$, $\lambda_\mu\in\bC$, $\mu\in\Lambda^p$ with $s(\mu)\in H$ such that
\[
\|\sum\lambda_\mu s_\mu s_\mu^*-s_v\|=\|\sum\lambda_\mu s_\mu s_\mu^*-\sum_{\nu\in v\Lambda^p}s_\nu s_\nu^*)\|<\epsilon.
\]
To the contrary, assume that $v\not\in H$. Then there is $\nu_0\in v\Lambda^p$ with $s(\nu_0)\not\in H$ as $H$ is $G$-saturated. So
$1\le \|(\sum\lambda_\mu s_\mu s_\mu^*-\sum_{\nu_0\ne\nu\in v\Lambda^p}s_\nu s_\nu^*)-s_{\nu_0}s_{\nu_0}^*\|<\epsilon$, an absurd.
Therefore, $v\in H$.
\end{proof}

\begin{lem}\label{L:quotient}
Let $(G,\Lambda)$ be a self-similar $k$-graph, and $H$ be a $G$-hereditary and $G$-saturated subset of $\Lambda^0$. Then $(G,\Lambda(\Lambda^0 \setminus H))$ is a self-similar $k$-graph. Denote by $\{s_\mu,u_g\}_{\mu\in\Lambda,g \in G}$ and $\{t_\mu,w_g\}_{\mu \in \Lambda(\Lambda^0 \setminus H),g \in G}$ the generators of $\ca_u(G,\Lambda)$ and $\ca_u(G,\Lambda(\Lambda^0 \setminus H))$, respectively.
Then there exists an isomorphism $\psi:\mathcal{O}_{G,\Lambda(\Lambda^0 \setminus H)} \to \mathcal{O}_{G,\Lambda} / I(H)$ such that $\psi(t_\mu w_g t_\nu^*)=s_\mu u_g s_\nu^*+I(H)$ for all $\mu,\nu \in \Lambda,g \in G$ with $s(\mu)=g \cdot s(\nu) \in \Lambda^0 \setminus H$.
\end{lem}
%%%%%%%%%%%%%%%%%

\begin{proof}
The first statement is easy to verify. We prove the second statement. For $g\in G$ and $\mu \in \Lambda$ and $g\in G$, define $W_g:= u_g+I(H)$ and $T_\mu:=s_\mu+I(H)$.
Then $\{W, T\}$ is a self-similar $(G, \Lambda(\Lambda^0\setminus H))$-family in $\ca_u(G,\Lambda)/I(H)$.
The universal property of $\O_{G,\Lambda (\Lambda^0\setminus H)}$ yields a surjective homomorphism
$\psi:\mathcal{O}_{G,\Lambda(\Lambda^0 \setminus H)} \to \mathcal{O}_{G,\Lambda} / I(H)$.
%%%%%%%%%
%%%%%%
Since $I(H)$ is gauge-invariant, there exists a strongly continuous homomorphism $\alpha:\bT^k \to \Aut(\mathcal{O}_{G,\Lambda} / I(H))$ such that $\alpha_z(\psi(t_\mu w_g t_\nu^*))=z^{d(\mu)-d(\nu)}\psi(t_\mu w_g t_\nu^*)$ for all $z \in \bT^k,\ \mu,\ \nu \in \Lambda(\Lambda^0 \setminus H),\ g \in G$ with $s(\mu)=g \cdot s(\nu)$. For brevity, by Theorem~\ref{T:OOG}, we identify $\O_{G,\Lambda}$ and $\O_{G,\Lambda(\Lambda^0\setminus H)}$ with $\O_{X_{G,\Lambda}}$ and $\O_{X_{G,\Lambda(\Lambda^0\setminus H)}}$ respectively.
Then by Proposition~\ref{P:GUIT} , in order to show that $\psi$ is an isomorphism, it remains to show that $\psi$ is injective on $C_0(\Lambda^0 \setminus H) \rtimes G$. By \cite[Proposition~4.1.9, Theorem~4.2.6]{BO08} there exists a faithful expectation $E:C_0(\Lambda^0 \setminus H) \rtimes G\to C_0(\Lambda^0 \setminus H)$ such that $E(t_vw_g)=\delta_{g,1_G}t_v$ for all $v \in \Lambda^0 \setminus H,\ g \in G$. Since $I(H)$ is diagonal-invariant, there exists a bounded linear map $L:\mathcal{O}_{G,\Lambda} / I(H) \to \mathcal{D}_\Lambda/I(H)$ such that $L(s_\mu u_g s_\nu^*+I(H))=\delta_{\mu,\nu} \delta_{g,1_G} s_\mu s_\mu^*+I(H)$ for all $g\in G$ and $(\mu,\nu) \in \Lambda \tensor[_g]\times{_s}\Lambda$. By Lemma~\ref{L:HIH}, $\psi \vert_{C_0(\Lambda^0 \setminus H)}$ is injective. Since $\psi \circ E=L \circ \psi$, by Proposition~\cite[Proposition 3.11]{Kat03}, $\psi$ is injective on $C_0(\Lambda^0 \setminus H) \rtimes G$. Hence $\psi$ is an isomorphism.
\end{proof}

\begin{thm}\label{T:chagi}
Let $(G,\Lambda)$ be a self-similar $k$-graph. Then there exists a one-to-one correspondence between the set of all $G$-hereditary and $G$-saturated subsets of $\Lambda^0$ and the set of all gauge- and diagonal-invariant ideals of $\mathcal{O}_{G,\Lambda}$ such that $H \mapsto I(H)$ with the inverse $I \mapsto H(I)$.
\end{thm}

\begin{proof}
First of all, fix a $G$-hereditary and $G$-saturated subset $H$ of $\Lambda^0$. Then Lemma~\ref{L:HIH} yields that $H(I(H)) = H$.

Fix a gauge- and diagonal-invariant ideal $I$ of $\mathcal{O}_{G,\Lambda}$. Suppose that $\{u, s\}$ and $\{w, t\}$ be the universal self-similar $(G,\Lambda)$-family and $(G,\Lambda(\Lambda^0 \setminus H(I)))$-family, respectively.
Clearly, $I(H(I)) \subseteq I$. So there is surjective quotient map $q:\mathcal{O}_{G,\Lambda}/I(H(I)) \to \mathcal{O}_{G,\Lambda} /I$.
By Lemma~\ref{L:quotient}, there exists a surjective homomorphism $\pi:\mathcal{O}_{G,\Lambda(\Lambda^0 \setminus H(I))} \to \mathcal{O}_{G,\Lambda} /I$ such that $\pi(t_\mu w_g t_\nu^*)=s_\mu u_g s_\nu^*+I$ for all $g\in G,\ \mu,\ \nu \in \Lambda(\Lambda^0 \setminus H(I))$ with $s(\mu)=g \cdot s(\nu)$. Since $I$ is gauge-invariant, there exists a strongly continuous homomorphism $\alpha:\bT^k \to \Aut(\mathcal{O}_{G,\Lambda} / I)$ such that $\alpha_z(\pi(t_\mu w_g t_\nu^*))=z^{d(\mu)-d(\nu)}\pi(t_\mu w_g t_\nu^*)$ for all $z \in \bT^k, \mu,\nu \in \Lambda(\Lambda^0 \setminus H(I)),g \in G$ with $s(\mu)=g \cdot s(\nu)$. As before, by Theorem~\ref{T:OOG} and Proposition~\ref{P:GUIT}, in order to show that $\pi$ is an isomorphism, one only needs to  show that $\pi$ is injective on $C_0(\Lambda^0 \setminus H(I)) \rtimes G$. By \cite[Proposition~4.1.9, Theorem~4.2.6]{BO08} there exists a faithful expectation $E:C_0(\Lambda^0 \setminus H(I)) \rtimes G\to C_0(\Lambda^0 \setminus H(I))$ such that $E(t_v w_g)=\delta_{g,1_G}t_v$ for all $v \in \Lambda^0 \setminus H(I),g \in G$. Since $I$ is diagonal-invariant, there exists a bounded linear map $L:\mathcal{O}_{G,\Lambda} / I \to \mathcal{D}_\Lambda/I$ such that $L(s_\mu u_g s_\nu^*+I)=\delta_{\mu,\nu} \delta_{g,1_G} s_\mu s_\mu^*+I$ for all $\mu,\nu \in \Lambda,g \in G$ with $s(\mu)=g \cdot s(\nu)$. By the definition of $H(I)$, one has that $\pi \vert_{C_0(\Lambda^0 \setminus H(I))}$ is injective. Since $\pi \circ E=L \circ \pi$, by \cite[Proposition 3.11]{Kat03}, $\pi$ is injective on $C_0(\Lambda^0 \setminus H(I)) \rtimes G$. Hence $\pi$ is an isomorphism, and we are done.
\end{proof}

Theorem~\ref{T:chagi} gives a clear relationship between $G$-hereditary $G$-saturated subsets and both gauge- and diagonal-invariant ideals. The corresponding result for $k$-graph C*-algebra results (see, e.g., \cite{RSY03}) says that the hereditary and saturated subsets are one-to-one corresponding to gauge-invariant ideals. This in particular implies that every gauge-invariant ideal of a $k$-graph C*-algebra is automatically diagonal-invariant. However, the example below shows that this is not true for self-similar $k$-graph C*-algebras any more in general.

\begin{eg}\label{gauge-inv not diag-inv}
Let $G:=\mathbb{Z}_2$ be the cyclic group of order $2$, and $\Lambda$ be a $k$-graph with a single vertex $v$. Define
$g \cdot \mu:=\mu$ and $g \vert_\mu:=g$ for all $g\in \{0,1\}$ and $\mu \in \Lambda$.
Then it is easy to check that $(G,\Lambda)$ is a pseudo free self-similar $k$-graph.
Since the only $G$-hereditary and $G$-saturated sets are $\mt$ and $\{v\}$, we deduce that the only gauge- and diagonal-invariant ideals of $\mathcal{O}_{G,\Lambda}$ ($\cong \O_\Lambda\otimes \ca(G)$) are $0$ and $\mathcal{O}_{G,\Lambda}$. On the other hand, let $p:=(u_0+u_1)/2$ which is a central projection, and let $I$ be the ideal of $\mathcal{O}_{G,\Lambda}$ generated by $p$. Then $I$ is gauge-invariant. However, since $p$ is central, $0\ne I=p(\mathcal{O}_{G,\Lambda})p \ne \O_{G,\Lambda}$. So $E_{G,\Lambda}(p)=u_0/2  \notin I$. Hence $I$ is not diagonal-invariant.
\end{eg}

Observe that $(G,\Lambda)$ in the above example is not strongly locally faithful. In the following, we show that, for a locally faithful self-similar $k$-graph $(G,\Lambda)$, every gauge-invariant ideal of $\O_{G,\Lambda}$ is also diagonal-invariant. This reconciles with the case of $k$-graphs (i.e., $G$ is trivial).

\begin{prop}\label{characterize gauge-inv ideal strongly faithful}
Let $(G,\Lambda)$ be a strongly locally faithful self-similar $k$-graph. Then every gauge-invariant ideal of $\mathcal{O}_{G,\Lambda}$ is diagonal-invariant. Hence there exists a one-to-one correspondence between the set of all $G$-hereditary and $G$-saturated subsets of $\Lambda^0$ and the set of all gauge-invariant ideals of $\mathcal{O}_{G,\Lambda}$ such that $H \mapsto I(H)$ with the inverse $I \mapsto H(I)$.
\end{prop}

\begin{proof}
Fix a gauge-invariant ideal $I$ of $\mathcal{O}_{G,\Lambda}$. Fix $a \in I$ and fix $\epsilon>0$. Let $E:\mathcal{O}_{G,\Lambda} \to \mathcal{O}_{G,\Lambda}^\gamma$ be the expectation given in \ref{SSS:gauge}. Then $\E_{G,\Lambda}(a)=\E_{G,\Lambda}(E(a))$. Since $I$ is gauge-invariant, $E(a) \in I$. Since $(G,\Lambda)$ is strongly locally faithful, for $g\in G$ and $(\mu,\nu)\in\Lambda\tensor[_g]\times{_s}\Lambda$, there exists $p \in \mathbb{N}^k$, such that $g \cdot \alpha \neq \alpha$ for all $\alpha \in s(\mu)\Lambda^p$. Then $s_\mu u_g s_\mu^*=\sum_{\alpha \in s(\mu)\Lambda^p}s_{\mu(g \cdot \alpha)}u_{g \vert_\alpha} s_{\mu\alpha}^*$. Thus there exists a finite sum $b=\sum_{i=1}^{n}\lambda_i s_{\mu_i} s_{\mu_i}^*+\sum_{i=1}^{n'}\lambda'_i s_{\mu'_i} u_{g_i}s_{\nu'_i}^*$ satisfying
\begin{itemize}
\item $d(\mu_i)=d(\mu'_j)=d(\nu'_j)$ for all $1 \leq i \leq n,1 \leq j \leq n'$;
\item $\mu_i \neq \mu_j$ for all $1 \leq i \neq j \leq n$;
\item $\mu'_i \neq \nu'_i$ for all $1 \leq i  \leq n'$;
\item $\Vert b-F(a)\Vert<\epsilon$.
\end{itemize}
Then
\[
\Vert \sum_{i=1}^{n}\lambda_i s_{\mu_i} s_{\mu_i}^*-\E_{G,\Lambda}(E(a)) \Vert=\Vert \E_{G,\Lambda}(b)-\E_{G,\Lambda}(E(a))\Vert\leq\Vert b-E(a)\Vert<\epsilon,
\]
and for $1 \leq i \leq n$, we have
\[
\Vert \lambda_i s_{\mu_i} s_{\mu_i}^*-s_{\mu_i}s_{\mu_i}^*F(a)s_{\mu_i}s_{\mu_i}^*\Vert=\Vert s_{\mu_i}s_{\mu_i}^*(b-F(a))s_{\mu_i}s_{\mu_i}^*\Vert \leq \Vert b-F(a)\Vert<\epsilon.
\]
So
\begin{align*}
&\Vert \E_{G,\Lambda}(a)-\sum_{i=1}^{n}s_{\mu_i}s_{\mu_i}^*E(a)s_{\mu_i}s_{\mu_i}^*\Vert
\\&\leq\Vert \E_{G,\Lambda}(E(a))-\sum_{i=1}^{n}\lambda_i s_{\mu_i} s_{\mu_i}^*\Vert+\Vert\sum_{i=1}^{n}(\lambda_i s_{\mu_i} s_{\mu_i}^*-s_{\mu_i}s_{\mu_i}^*E(a)s_{\mu_i}s_{\mu_i}^*)\Vert
\\&<\epsilon+\Vert \sum_{i=1}^{n}s_{\mu_i} s_{\mu_i}^*(\lambda_i s_{\mu_i} s_{\mu_i}^*-E(a)s_{\mu_i}s_{\mu_i}^*)s_{\mu_i}s_{\mu_i}^*\Vert
\\&<2\epsilon.
\end{align*}
Since $\sum_{i=1}^{n}s_{\mu_i}s_{\mu_i}^*E(a)s_{\mu_i}s_{\mu_i}^* \in I$ and $\epsilon$ is arbitrary, we conclude that $\E_{G,\Lambda}(a)$ $\in I$. Therefore $I$ is diagonal-invariant.

The second statement follows directly from the first statement and Theorem~\ref{T:chagi}.
\end{proof}

%%%%%%%%%%%%%%%%

\section{Primitive Ideals of Self-Similar $k$-Graph C*-Algebras}
\label{S:primitive}

Throughout this section, let $(G,\Lambda)$ be a self-similar $k$-graph which satisfies
\begin{align*}
\tag{FV}
&\textsf{$g\cdot v=v\qforal g\in G\text{ and }v\in \Lambda^0$};\\
\tag{Cyc}
&\textsf{every cycline triple of $(G,\Lambda T)$ is a cycline pair of $(G,\Lambda T)$ for any maximal tail $T$ of $\Lambda$.}
 \end{align*}
Here, we use that the simple fact that $(G,\Lambda T)$ is a pseudo free self-similar $k$-graph for any maximal tail $T$. In this section, we identify all primitive ideals of $\O_{G,\Lambda}$ and study some basic properties of the Jacobson topology of its primitive ideal space. Our strategy of finding primitive ideals follows the vein of \cite{CKSS14}.
%\begin{notation}

For a maximal tail $T$ of $\Lambda$, denote by $H_T$ the subset of $T$ which consists of all $v \in T$ satisfying the following property: For any $p, q \in \mathbb{N}^k$ with $p-q \in \Per_{G,\Lambda T}$ and for any $\mu \in v\Lambda^pT$, there exists a unique $\nu \in v\Lambda^qT$ such that $(\mu,\nu)$ is a cycline pair of $\Lambda T$.
%\end{notation}

\begin{prop}\label{Per_G,Lambda T is subgroup}
Let $T$ be a maximal tail of $\Lambda$. Then we have the following properties.
\begin{enumeratei}
\item\label{unique cycline pair} $H_{T}$ is a nonempty hereditary subset of $\Lambda T$;

\item\label{self-similar k graph (G,H_T Lambda)} $(G,H_T \Lambda T)$ is a pseudo free self-similar $k$-graph such that any cycline triple of $(G,H_T \Lambda T)$ is a cycline pair of $(G,H_T \Lambda T)$; 

\item\label{Per_G,Lambda T} $\Per_{G,H_T \Lambda T}=\Per_{H_T \Lambda T}$ is a subgroup of $\mathbb{Z}^k$;

\item\label{Z^k/ Per_G,H_T Lambda T-graph} Define an equivalence relation on $H_T \Lambda T$ as follows: for any $\mu,\nu \in H_T \Lambda T$, define $\mu \sim_T \nu$ if $(\mu,\nu)$ is a cycline pair of $H_T \Lambda T$. Then $(G,H_T \Lambda T/\!\!\sim_T)$ is an aperiodic, pseudo free self-similar $\mathbb{N}^k / \Per_{H_T \Lambda T}$-graph;

\item Let $\{w_z\}_{z \in \Per_{G,H_T \Lambda T}}$ be the generators of $\ca(\Per_{H_T\Lambda T})$. Then there exists an injective homomorphism from $\ca(\Per_{H_T \Lambda T})$ into $ZM(\mathcal{O}_{G,H_T \Lambda T})$ sending $w_z$ to $W_z$, where
\[
W_z=\sum_{z=p-q,\, (\mu,\nu) \in\C_{H_T\Lambda T}^{p,q}}s_\mu s_\nu^*.
\]
\end{enumeratei}
\end{prop}

\begin{proof}
(i) This is \cite[Theorem~4.2 (2)]{CKSS14}.

(ii) This is true due to the assumption that any cycline triple of $(G,\Lambda T)$ is a cycline pair of $(G,\Lambda T)$.

(iii) From (ii) above, one has $\Per_{G,H_T \Lambda T}=\Per_{H_T \Lambda T}$. Then (iii) follows from \cite[Theorem~4.2 (1)]{CKSS14}.

(iv) 
For any $g\in G$, $\mu\sim_T\nu$ in $H_T \Lambda T$ and $x \in s(g \cdot \mu)(H_T \Lambda T)^\infty$, we compute that
\[
(g \cdot \mu) x=g \cdot (\mu(g\vert_\mu^{-1} \cdot x))=g \cdot (\nu(g\vert_\mu^{-1} \cdot x))=(g \cdot \nu) (g \vert_\nu g\vert_\mu^{-1} \cdot x).
\]
So $(g \cdot \nu,g \vert_\nu g \vert_\mu^{-1},g \cdot \mu)$ is a cycline triple of $(G,H_T\Lambda T)$. From (ii), we deduce that
\[
g \cdot \nu\sim_T g \cdot \mu\text{ and }g \vert_\mu=g \vert_\nu.
\]
If $[\mu]$ stands for the equivalence class of $\mu\in H_T\Lambda T$, then the above shows that 
\begin{align}
\label{E:gmu}
g\cdot [\mu]:=[g\cdot \mu]\text{ and } g|_{[\mu]}:=g|_\mu
\end{align}
are well-defined. It is now easy to verify that $(G,H_T \Lambda T/\!\! \sim_T)$ with the two operations defined in \eqref{E:gmu} is a self-similar $\mathbb{N}^k / \Per_{H_T \Lambda T}$-graph. 
Moreover it is also pseudo free. Indeed, $g\cdot[\mu]=[\mu]\text{ and }g|_{[\mu]}=1_G\Rightarrow g\cdot \mu \sim \mu \text{ and } g|_\mu =1_G\Rightarrow g\cdot \mu= \mu \text{ and } g|_\mu =1_G\Rightarrow g=1_G$ as $(G,\Lambda)$ is pseudo free. 

%%%%%%%%%%%%%%
%For any cycline pair $(\mu,1_G,\nu)$ of $(G,H_T \Lambda T)$ and $x \in s(g \cdot \mu)(H_T \Lambda T)^\infty$, we compute that
%\[
%(g \cdot \mu) x=g \cdot (\mu(g\vert_\mu^{-1} \cdot x))=g \cdot (\nu(g\vert_\mu^{-1} \cdot x))=(g \cdot \nu) (g \vert_\nu g\vert_\mu^{-1} \cdot x).
%\]
%So $(g \cdot \nu,g \vert_\nu g \vert_\mu^{-1},g \cdot \mu)$ is a cycline triple of $(G,H_T\Lambda T)$. By (ii), we deduce that $(g \cdot \nu,1_G,g \cdot \mu)$ is a cycline pair of $(G,H_T \Lambda T)$ and $g \vert_\mu=g \vert_\nu$. Hence $(G,H_T \Lambda T/\!\! \sim_T)$ is a pseudo free self-similar $\mathbb{N}^k / \Per_{H_T \Lambda T}$-graph.
%%%%%%%%%%%%%%

In the sequel, we show that $H_T \Lambda T/\!\! \sim_T$ is aperiodic. For this, let us fix
$\mu \in H_T \Lambda T, p,q \in \mathbb{N}^k,g \in G$ with $p - q \notin \Per_{H_T \Lambda T}$ or $g \neq 1_G$. We split into two cases.

\underline{Case 1: $p-q \notin \Per_{H_T \Lambda T}$.} We first claim that there exists $x \in s(\mu) (H_T\Lambda T)^\infty$ such that $\sigma^p(x) \neq g \vert_{(\mu x)(q,q+d(\mu))} \cdot \sigma^q(x)$.
To the contrary, take an arbitrary $x \in s(\mu) (H_T\Lambda T)^\infty$. Then as the proof of Case 1 for \cite[Proposition 3.6]{LY18}, one gets that $(g \vert_{(\mu x)(q,q+d(\mu))} \cdot x(q,p+q),g \vert_{(\mu x)(q,q+p+d(\mu))},x(p,p+q))$ is a cycline triple of $(G,H_T \Lambda T)$. So $p-q \in  \Per_{H_T \Lambda T}$, which is a contradiction.

Let now $x$ be such an infinite path. Find $l \in \mathbb{N}^k$ such that $\sigma^p(x)(0,l) \neq g \vert_{(\mu x)(q,q+d(\mu))} \cdot \sigma^q( x) (0,l)$. Let $\nu:=x(0,p+q+l)$. Then $d(\nu)-p-q \geq 0$ and $\nu(p,d(\nu)-q) \neq g \vert_{(\mu\nu)(q,d(\mu)+q)} \cdot \nu(q,d(\nu)-p)$. By \cite[Lemma 4.1]{Yan15}, one has
$[\nu(p,d(\nu)-q)] \neq [g \vert_{(\mu\nu)(q,d(\mu)+q)} \cdot \nu(q,d(\nu)-p)]$

\underline{Case 2: $p-q \in \Per_{H_T \Lambda T}$.} Then $g \neq 1_G$. LEt $\nu \in s(\mu)(H_T \Lambda T)$ be such that $d(\nu) \geq p$ and $g \vert_{(\mu \nu)(p,d(\mu)+p)} \neq 1_G$. Denote by $h:=g \vert_{(\mu \nu)(p,d(\mu)+p)}$. Suppose that for any $\gamma \in s(\nu)\Lambda$, we have $\nu(p,d(\nu))\gamma=h \cdot (\nu(p,d(\nu))\gamma)$. Then $(h \cdot \nu(p,d(\nu)),h \vert_{\nu(p,d(\nu))}, \nu(p,d(\nu)))$ is a cycline triple of $(G,H_T \Lambda T)$. Since $(G,H_T \Lambda T/\!\! \sim_T)$ is pseudo free, $h =1_G$ which is a contradiction. So there exists $\gamma \in s(\nu)\Lambda$, such that $\nu(p,d(\nu))\gamma \neq h \cdot (\nu(p,d(\nu))\gamma)$, and so $[\nu(p,d(\nu))\gamma] \neq [h \cdot (\nu(p,d(\nu))\gamma)]$ by \cite[Lemma 4.1]{Yan15} again.

Therefore by Proposition~\ref{P:G-ape} $(G,H_T \Lambda T/\!\! \sim_T)$ is aperiodic.

(v) Lemma~\ref{operators in multiplier} and a similar argument of the one in \cite[Theorem~4.9]{LY18} give a homomorphism $i:\ca(\Per_{H_T\Lambda T}) \to ZM(\mathcal{O}_{G,H_T\Lambda T})$ such that $i(w_n)=W_n$ for all $n \in \Per_{G,H_T\Lambda T}$. By \cite[Proposition~4.1.9]{BO08}, there exists a faithful expectation $E_1:\ca(\Per_{G,H_T\Lambda T}) \to \mathbb{C}$ such that $E_1(v_0)=1$ and $E_1(v_n)=0$ for all $0\ne n \in \Per_{G,H_T\Lambda T}$. As in the proof of Theorem~\ref{T:OOG}, there exists an expectation $E_2:\mathcal{O}_{G,H_T\Lambda T} \to \mathcal{D}_{H_T \Lambda T}$ such that $E_2(s_\mu u_g s_\nu^*)=\delta_{\mu,\nu}\delta_{g,1_G}s_\mu s_\mu^*$ for all $\mu,\nu \in H_T\Lambda T,g \in G$ with $s(\mu)=s(\nu)$. And $E_2$ extends to an expectation from $M(\mathcal{O}_{G,H_T\Lambda T})$ onto $M(\mathcal{D}_{H_T\Lambda T})$ such that $E_2(S)(a)=E_2(S(a))$ for all $S \in M(\mathcal{O}_{G,H_T\Lambda T}), a \in \mathcal{D}_{H_T\Lambda T}$. It is straightforward to check that $i\circ E_1=E_2\circ i$ and $i$ is injective on $w_0 \mathbb{C}$. By \cite[Proposition~3.11]{Kat03}, $i$ is injective.
\end{proof}

For any infinite path $x\in\Lambda^\infty$, let
\[
[x]:=\{y \in \Lambda^\infty:\text{there are } p,q \in \mathbb{N}^k \text{ and } g\in G \text{ such that } \sigma^p(x)=g \cdot\sigma^q( y)\}.
\]

\begin{lem}[{\cite[Lemma~5.2]{CKSS14}}]\label{L:cof}
Let $T$ be a maximal tail of $\Lambda$. Then there exists a cofinal infinite path in $\Lambda T$.
\end{lem}

%%%%%%%%%%%%%%%%%%%revised
\begin{prop}\label{P:irrfx}
Let $T$ be a maximal tail of $\Lambda$, $f \in \widehat{\mathbb{Z}^k}$, and $x$ be a cofinal infinite path in $\Lambda T$ (see Lemma~\ref{L:cof}). Then there exists an irreducible representation $\pi_{f,x}$ of $\O_{G,\Lambda}$ on $\ell^2([x])$.
\end{prop}

\begin{proof}
For $g\in G$ and $\mu\in \Lambda$, let $U_g$ and $S_\mu$ be two operators on $\ell^2([x])$ determined by
\begin{align*}
U_g(\delta_y)&=\delta_{g \cdot y},\\
S_\mu(\delta_y)&=\begin{cases}
   f(d(\mu))\delta_{\mu y}   &\text{ if $s(\mu)=y(0,0)$} \\
   0   &\text{ otherwise }.
\end{cases}
\end{align*}
for all $\mu \in \Lambda,\, g \in G, \, y \in [x]$.
Then $\{U, S\}$ is a self-similar $(G,\Lambda)$-family in $B(\ell^2([x]))$. Thus there exists a homomorphism $\pi_{f,x}:\O_{G,\Lambda} \to B(\ell^2([x]))$ such that
 $\pi_{f,x}(s_\mu u_g s_\nu^*)=S_\mu U_g S_\nu^*$ for all $g\in G$ and $(\mu,\nu)\in\Lambda\tensor[_g]\times{_s}\Lambda$.
 In order to prove that $\pi_{f,x}$ is irreducible, by the double commutant theorem, it is enough to show that rank-one operators in $B(\ell^2([x]))$ are in the strong closure of $\pi_{f,x}(\mathcal{O}_{G,\Lambda})$. Fix $y,z \in [x]$. Since $\theta_{\delta_y,\delta_z}=\theta_{\delta_y,\delta_x}\theta_{\delta_x,\delta_z}$, it is sufficient to show that $\theta_{\delta_x,\delta_y}$ can be approximated by a bounded net in $\pi_{f,x}(\mathcal{O}_{G,\Lambda})$ in the strong operator topology. By the definition of $[x]$, there exist $p,q \in \mathbb{N}^k,g\in G$, such that $\sigma^p(x)=g \cdot\sigma^q( y)$. It is straightforward to see that $(\pi_{f,x}(s_{x(0,p+l)}u_{g\vert_{y(q,q+l)}} s_{y(0,q+l)}^*))_{l \in \mathbb{N}^k} \xrightarrow{\text{SOT}} f(p-q)\theta_{\delta_x,\delta_y}$. So $\pi_{f,x}|_{\mathcal{O}_{G,\Lambda}}$ is irreducible.
\end{proof}

\begin{prop}\label{P:Ifx}
Let $I$ be a primitive ideal of $\mathcal{O}_{G,\Lambda}$. Then there exist a maximal tail $T$ of $\Lambda$ and $f \in \widehat{\Per}_{G,H_T \Lambda T}$ such that for any cofinal infinite path $x$ in $\Lambda T$ and for any character $\widetilde{f}$ of $\mathbb{Z}^k$ extending $f$ (see \cite[Corollary~4.41]{Fol95}), we have $I=\ker(\pi_{\widetilde{f},x})$.
\end{prop}

\begin{proof}
Firstly, let $\pi:\mathcal{O}_{G,\Lambda} \to B(H)$ be an irreducible representation such that $\ker \pi = I$. Denote by $\overline{\pi}:\mathcal{O}_{G,\Lambda}/I \to B(H)$ the injective quotient representation. Define $T:=\{v \in \Lambda^0: s_v \notin I\}$. For any $v \in T$ and $\mu \in \Lambda v$, we have $\pi(s_\mu^* s_\mu)=\pi(s_v ) \neq 0$. So $\pi(s_\mu) \neq 0$. Hence $\pi(s_\mu s_\mu^*) \neq 0$. It follows that $\pi(s_{r(\mu)}) \neq 0$ and $r(\mu) \in T$.
For any $v \in T, p\in \mathbb{N}^k$, there exists $\mu \in v \Lambda^p$ such that $\pi(s_\mu s_\mu^*) \neq 0$. So $\pi(s_{s(\mu)})=\pi(s_\mu^* s_\mu) \neq 0$. Hence $s(\mu) \in T$.
 For any $v_1,v_2 \in T$, denote by $I(v_1)$ the ideal of $\mathcal{O}_{G,\Lambda}$ generated by $s_{v_1}$. Since $\pi$ is irreducible, $\overline{\mathrm{span}}(\pi(I(v_1))H)=H$. Since $v_2 \in T$, there exist $\mu,\nu \in \Lambda,g \in G$ such that $\pi(s_{v_2}s_\mu u_g s_{\nu}^* s_{v_1}) \neq 0$. Notice that $s(\mu)=s(\nu)$ due to the assumption that the action fixes vertices. So $v_1 \Lambda s(\mu), v_2 \Lambda s(\mu) \neq \mt$.
Hence $T$ satisfies Definition~\ref{D:MT} (i) - (iii), and so it is a maximal tail.

Notice that $\Lambda^0 \setminus T$ is $G$-hereditary and $G$-saturated because $T$ is a maximal tail and the action fixes vertices. Let $I(\Lambda^0 \setminus T)$ be the ideal of $\mathcal{O}_{G,\Lambda}$ generated by $\{s_v : v \notin T\}$. Then $I(\Lambda^0 \setminus T) \subseteq I$. Denote by $q_1:\mathcal{O}_{G,\Lambda}/ I(\Lambda^0 \setminus T) \to \mathcal{O}_{G,\Lambda}/I$ the quotient map. Recall from Lemma~\ref{L:quotient} that $\mathcal{O}_{G,\Lambda T}$ is isomorphic to $\mathcal{O}_{G,\Lambda} /I(\Lambda^0 \setminus T)$ via $\psi$. So $\overline{\pi} \circ q_1\circ \psi$ is an irreducible representation of $\mathcal{O}_{G,\Lambda T}$.

Let $I(H_T)$ be the ideal of $\mathcal{O}_{G,\Lambda T}$ generated by $\{s_v:v \in H_T\}$. Since $\overline{\pi} \circ q_1\circ \psi(I(H_T)) \neq 0$, by \cite[Proposition~A.26]{RW98} $\overline{\pi} \circ q_1\circ \psi\vert_{I(H_T)}=:\tilde\pi_1$ is irreducible. By Lemma~\ref{L:emb}, $\mathcal{O}_{G,H_T \Lambda T}$ is a full corner of $I(H_T)$. By the Rieffel correspondence we obtain an irreducible representation $\rho_1:\mathcal{O}_{G,H_T \Lambda T} \to B(H_1)$ such that $\ker(\rho_1)=\ker(\tilde\pi_1) \cap \mathcal{O}_{G,H_T \Lambda T}$ (we do not need the formulation of $\rho_1$ and one may refer the material of the Rieffel correspondence to \cite{RW98}). Denote by $\widetilde{\rho_1}:M(\mathcal{O}_{G,H_T \Lambda T}) \to B(H_1)$ the unique extension of $\rho_1$ which is also irreducible. Since $\ca(\Per_{H_T \Lambda T})$ embeds into $ZM(\mathcal{O}_{G,H_T \Lambda T})$ by Proposition~\ref{Per_G,Lambda T is subgroup}, we deduce that $\widetilde{\rho_1} \vert_{\ca(\Per_{H_T \Lambda T})}$ is irreducible. By identifying $\ca(\Per_{H_T \Lambda T})$ with $C(\widehat{\Per}_{H_T \Lambda T})$ via the Fourier transformation, we can find a unique $f \in \widehat{\Per}_{H_T \Lambda T}$ such that $\widetilde{\rho_1}(F)=F(f) \cdot 1_{B(H_1)}$ for all $F \in C(\widehat{\Per}_{H_T \Lambda T})$.

Take an arbitrary character $\widetilde{f}$ of $\mathbb{Z}^k$ extending $f$, and take an arbitrary cofinal infinite path $x$ in $\Lambda T$. By Proposition~\ref{P:irrfx}, we obtain an irreducible representation $\pi_{\widetilde{f},x}:\mathcal{O}_{G,\Lambda} \to B(\ell^2([x]))$. Denote by $\overline{\pi_{\widetilde{f},x}}:\mathcal{O}_{G,\Lambda}/\ker(\pi_{\widetilde{f},x}) \to B(\ell^2([x]))$ the injective quotient representation. It is straightforward to see that $I(\Lambda^0 \setminus T) \subseteq \ker(\pi_{\widetilde{f},x})$. Denote by $q_2:\mathcal{O}_{G,\Lambda}/ I(\Lambda^0 \setminus T) \to \mathcal{O}_{G,\Lambda}/\ker(\pi_{\widetilde{f},x})$ the quotient map. Then $\overline{\pi_{\widetilde{f},x}} \circ q_2\circ \psi$ is an irreducible representation of $\mathcal{O}_{G,\Lambda T}$. By the cofinal property of $x$, we have $\overline{\pi_{\widetilde{f},x}} \circ q_2\circ\psi(I(H_T)) \neq 0$. So by \cite[Proposition~A.26]{RW98} $\overline{\pi_{\widetilde{f},x}} \circ q_2\circ\psi\vert_{I(H_T)}=:\tilde \pi_2$ is irreducible. Again by the Rieffel correspondence we obtain an irreducible representation $\rho_2:\mathcal{O}_{G,H_T \Lambda T} \to B(H_2)$ such that $\ker(\rho_2)=\ker(\tilde\pi_2) \cap \mathcal{O}_{G,H_T \Lambda T}$.

Let $\tilde I_{T,f}$ (resp. $\tilde J_{T,f}$) denote the ideal of $\mathcal{O}_{G, \Lambda T}$ (resp. $\mathcal{O}_{G, H_T\Lambda T}$) generated by $\{s_\mu-f(d(\mu)-d(\nu))s_\nu:(\mu,\nu)\in\C_{H_T \Lambda T}\}$.
Notice that $\tilde I_{T,f} \subseteq I(H_T)$ and $\tilde I_{T,f} \cap \mathcal{O}_{G,H_T \Lambda T}=\tilde J_{T,f}$. For any cycline pair $(\mu,1_G,\nu)$ of $(G,H_T \Lambda T)$, we have $(1/\widetilde{f}(d(\mu)))s_\mu-(1/\widetilde{f}(d(\nu)))s_\nu=(1/\widetilde{f}(d(\mu)))(s_\mu-f(d(\mu)-d(\nu))s_\nu) \in \tilde J_{T,f}$. So there exists a surjective homomorphism $h:\mathcal{O}_{G,H_T \Lambda T / \sim_T} \to \mathcal{O}_{G,H_T \Lambda T}/\tilde J_{T,f}$ such that $h(s_{[\mu]}u_g s_{[\nu]}^*)=f(d(\nu)-d(\mu))s_\mu u_g s_{\nu}^*+\tilde J_{T,f}$ for all $\mu,\, \nu \in H_T \Lambda T,\, g \in G$ with $s(\mu)=g \cdot s(\nu)$.

For any cycline pair $(\mu,1_G,\nu)$ of $(G,H_T \Lambda T)$, by the Fourier transformation, we compute that
\begin{align*}
\rho_1(s_\mu-f(d(\mu)-d(\nu))s_\nu)&=\widetilde{\rho_1}((W_{d(\mu)-d(\nu)}s_\nu-f(d(\mu)-d(\nu)) s_\nu)=0.
\end{align*}
So $\tilde J_{T,f} \subseteq \ker(\rho_1)$. Meanwhile, we calculate that for all $y \in s(\mu)[x]$,
\[
\overline{\pi_{\widetilde{f},x}} \circ q_2\circ \psi(s_\mu-f(d(\mu)-d(\nu))s_\nu)(\delta_y)=\widetilde{f}(d(\mu))(\delta_{\mu y}-\delta_{\nu y})=0.
\]
So $\tilde I_{T,f} \subseteq \ker(\tilde \pi_2)$. Hence $\tilde J_{T,f} \subseteq \ker(\rho_2)$. For $i=1,2$, denote by $\varphi_i:\mathcal{O}_{G,H_T \Lambda T}/\tilde J_{T,f} \to \mathcal{O}_{G,H_T \Lambda T}/\ker(\rho_i)$
 the quotient maps. It is straightforward to see that $\varphi_i \circ h(s_v) \neq 0$
 for all $v \in H_T$. By Theorem~\ref{T:aCKU} and Proposition~\ref{Per_G,Lambda T is subgroup} (iv), $\varphi_i\circ h$ is injective, and so is $\varphi_i$ $(i=1,2)$.
 Thus $J_{T,f}=\ker(\rho_i)$ $(i=1,2)$. The Rieffel correspondence yields that
 \begin{align}
 \label{E:ITf}
 \tilde I_{T,f}=\ker(\overline{\pi} \circ q_1\circ\psi\vert_{I(H_T)})
 =\ker(\overline{\pi_{\widetilde{f},x}} \circ q_2\circ\psi\vert_{I(H_T)}).
 \end{align}
Since $\overline{\pi} \circ q_1$ and $\overline{\pi_{\widetilde{f},x}} \circ q_2$ are irreducible, \cite[Proposition~A.26]{RW98} implies that $\ker(\overline{\pi} \circ q_1 )=\ker(\overline{\pi_{\widetilde{f},x}} \circ q_2 )$. Therefore $I=\ker(\pi_{\widetilde{f},x})$ as both $\overline{\pi}$ and $\overline{\pi_{\widetilde{f},x}}$ are injective.
\end{proof}

%%%%%%%%%%%%%%%%%

Combining Propositions~\ref{P:irrfx} and \ref{P:Ifx}, we are able to characterize all primitive ideals of $\mathcal{O}_{G,\Lambda}$.

The following theorem is a generalization of \cite[Corollary~5.4]{CKSS14}.

\begin{thm}\label{T:primcha}
There exists a bijection
\begin{align}
\label{E:Phi}
\Phi:\amalg_{\{T \subseteq \Lambda^0:T \text{ {\rm is a maximal tail}}\}}\{T\} \times \widehat{\Per}_{H_T \Lambda T} \to \mathrm{Prim}(\mathcal{O}_{G,\Lambda})
\end{align}
such that $\Phi(T,f)=\ker(\pi_{\widetilde{f},x})$ for any maximal tail $T$ of $\Lambda$, any $f \in \widehat{\Per}_{H_T \Lambda T}$, any character $\widetilde{f}$ of $\mathbb{Z}^k$ extending $f$, and any cofinal infinite path $x$ in $\Lambda T$.
\end{thm}

%\begin{notation}
For later use, we introduce the following notation:
\begin{align*}
M(\Lambda)&:=\{T\subseteq \Lambda^0: T\text{ is a maximal tail of }\Lambda\},\\
M_\gamma(\Lambda)&:=\{T \in M(\Lambda):T\text{ is aperiodic, i.e., }\Per_{H_T \Lambda T}=\{0\}\}, \\
M_\tau(\Lambda)&:=\{T \in M(\Lambda):T\text{ is periodic, i.e., }\Per_{H_T \Lambda T}\neq \{0\}\}.
\end{align*}
For any $T \in M(\Lambda)$, any $f \in \widehat{\Per}_{G,H_T \Lambda T}$, we simply denote by $I_{T,f}$ the primitive ideal $\Phi(T,f)$ obtained from Theorem~\ref{T:primcha}. For any $T \in M_\gamma(\Lambda)$, since $1$ is the only element in $\widehat{\Per}_{H_T \Lambda T}$, we denote $\Phi(T,1)$ as $I_{T,1}$ or just $I_T$.

\begin{cor}\label{C:con1}
For any $T \in M(\Lambda)$, $f,g \in \widehat{\Per}_{H_T \Lambda T}$ and $D\subseteq \widehat{\Per}_{H_T \Lambda T}$, we have
\begin{enumeratei}
\item $T \in M_{\gamma}(\Lambda)\implies I_T=I(\Lambda^0 \setminus T)$;

\item $\bigcap_{f' \in D}I_{T,f'} \subseteq I_{T,f}\iff f \in \overline{D}$, provided that $k <\infty$;

\item $I_{T,f} \subseteq I_{T,g}\iff f =g$, provided that $k <\infty$.
\end{enumeratei}
\end{cor}

\begin{proof}
(i) Clearly $H_T=T$ as $T \in M_{\gamma}(\Lambda)$. On one hand, $\psi(I_T/I(\Lambda^0 \setminus T)) \cap I(H_T)=\psi(I_T/I(\Lambda^0 \setminus T))$, and on the other hand
$\psi(I_T/I(\Lambda^0 \setminus T)) \cap I(H_T)=\tilde I_{T,f}$. So one has $T/I(\Lambda^0 \setminus T)=\tilde I_{T,f}$, which is $0$ as $T$ is aperiodic. So $I_T=I(\Lambda^0 \setminus T)$.

(ii) From the proof of Proposition~\ref{P:Ifx}, one has $I(\Lambda^0 \setminus T) \subseteq I_{T,f}$ for all $f\in \widehat{\Per}_{H_T \Lambda T}$. So it follows from \cite[Corollary~A.28]{RW98} that
\begin{align*}
\bigcap_{f' \in D}I_{T,f'} \subseteq I_{T,f}
&\iff\bigcap_{f' \in D}(I_{T,f'} /I(\Lambda^0 \setminus T)) \subseteq I_{T,f}/I(\Lambda^0 \setminus T)\\
&\iff \bigcap_{f' \in D}\tilde I_{T,f'} \subseteq \tilde I_{T,f} \ \text{ (by \eqref{E:ITf} and Proposition~\ref{P:Ifx})}\\    
&\iff \bigcap_{f' \in D}\tilde J_{T,f'} \subseteq \tilde J_{T,f} \ \text{ (by Lemma~\ref{L:emb})}.
\end{align*}
To show (ii) is equivalent to proving
\[
\bigcap_{f' \in D}J_{T,f'} \subseteq J_{T,f}\iff f \in \overline{D}.
\]
For this, first asume that $\bigcap_{f' \in D}\tilde J_{T,f'} \subseteq \tilde J_{T,f}$. Identifying $\ca(\Per_{H_T \Lambda T})$ with $C(\widehat{\Per}_{H_T \Lambda T})$, we have
$(\bigcap_{f' \in D}\tilde J_{T,f'}) \cap C(\widehat{\Per}_{H_T \Lambda T})\subseteq \tilde J_{T,f} \cap C(\widehat{\Per}_{H_T \Lambda T})$. So
$\bigcap_{f' \in D}C_0(\widehat{\Per}_{H_T \Lambda T} \setminus \{f'\}) \subseteq C_0(\widehat{\Per}_{H_T \Lambda T} \setminus \{f\})$. Hence $f \in \overline{D}$.

Conversely, suppose that $f \in \overline{D}$. Take an arbitrary $\widetilde{f} \in \widehat{\mathbb{Z}^k}$ extending $f$, and take an arbitrary cofinal infinite path $x$ in $\Lambda T$. Fix $a \in \bigcap_{f' \in D}\tilde J_{T,f'}$, fix $y \in [x]$, and fix $\epsilon >0$. Then there exist $\{\mu_i,\nu_i\}_{i=1}^{M} \subseteq H_T\Lambda T,\{g_i\}_{i=1}^{M} \subseteq G$ such that $\Vert \sum_{i=1}^{M}s_{\mu_i}u_{g_i}s_{\nu_i}^*-a\Vert<\epsilon/3$. Denote by $S:=\{1 \leq i \leq M:y(0,d(\nu_i))=\nu_i\}$. Since $f \in \overline{D}$, by \cite[Theorem~II.1.6]{Hun80} there exists $F \in \widehat{\mathbb{Z}^k}$ such that $\vert F(d(\mu_i)-d(\nu_i))-\widetilde{f}(d(\mu_i)-d(\nu_i)) \vert<\epsilon/(3(\vert S\vert+1))$ for all $i \in S$, and that $F \vert_{\widehat{\Per}_{H_T \Lambda T}} \in D$. Then $\Vert \sum_{i=1}^{M}\pi_{F,x}(s_{\mu_i}u_{g_i}s_{\nu_i}^*)(\delta_y)\Vert=\Vert \sum_{i \in S}F(d(\mu_i)-d(\nu_i))\delta_{\mu_i(g_i \cdot \sigma^{d(\nu_i)}(y))}\Vert<\epsilon/3$ because $a \in \bigcap_{f' \in D}\tilde J_{T,f'}$.
\begin{align*}
\Vert\pi_{\widetilde{f},x}(a)(\delta_y)\Vert&<\Vert \sum_{i=1}^{M}\pi_{\widetilde{f},x}(s_{\mu_i}u_{g_i}s_{\nu_i}^*)(\delta_y)\Vert+\epsilon/3
\\&=\Vert \sum_{i \in S}\widetilde{f}(d(\mu_i)-d(\nu_i))\delta_{\mu_i(g_i \cdot \sigma^{d(\nu_i)}(y))}\Vert+\epsilon/3
\\&\leq \Vert \sum_{i \in S}(\widetilde{f}(d(\mu_i)-d(\nu_i))-F(d(\mu_i)-d(\nu_i)))\delta_{\mu_i(g_i \cdot \sigma^{d(\nu_i)}(y))}\Vert
\\&\hskip .5cm +\Vert\sum_{i \in S}F(d(\mu_i)-d(\nu_i))\delta_{\mu_i(g_i \cdot \sigma^{d(\nu_i)}(y))}\Vert+\epsilon/3
\\&<\epsilon.
\end{align*}
So $\Vert\pi_{\widetilde{f},x}(a)(\delta_y)\Vert=0$. Hence $\pi_{\widetilde{f},x}(a)=0$. Therefore $a \in \tilde J_{T,f}$ as $\tilde J_{T,f}=\tilde I_{T,f}\cap \O_{G,H_T\Lambda T}\subseteq I(H_T)\cap \O_{G,H_T\Lambda T}$.

(iii) This follows directly from (ii) above .
\end{proof}

The following corollary is a generalization of \cite[Corollary~5.6]{CKSS14}.

\begin{cor}\label{characterize primitive}
$\O_{G,\Lambda}$ is primitive if and only if $\Lambda^0 \in M_\gamma(\Lambda)$.
\end{cor}

\begin{proof}
First suppose that $\mathcal{O}_{G,\Lambda}$ is primitive. Then the zero ideal is primitive. So there exist $T \in M(\Lambda)$ and $f \in \widehat{\Per}_{H_T \Lambda T}$ such that $I_{T,f}=0$. Since
$I(\Lambda^0 \setminus T) \subseteq I_{T,f}$, we have $T=\Lambda^0$. It follows from \eqref{E:ITf} that  $\tilde I_{T,f}$ is the zero ideal of $\mathcal{O}_{G,\Lambda T}$. Thus by the definition of $\tilde I_{T,f}$, one has  $\Per_{H_T \Lambda T}=0$. Hence $\Lambda^0 \in M_\gamma(\Lambda)$. The converse is an immediate consequence of Corollary~\ref{C:con1} (i).
\end{proof}

%%%%%%%%%%%%%%%%

\section{Some Examples}
\label{S:ex}

Let $(G,\Lambda)$ be a self-similar $k$-graph satisfying Conditions~(FV) and (Cyc).
As applications of the results of Section~\ref{S:primitive}, we completely describe the primitive idea space of $\O_{G,\Lambda}$ for two classes of self-similar $k$-graphs: (a) $k<\infty$ and $\Lambda^0$ is a maximal tail; and (b) $\Lambda$ is strongly aperiodic.

\subsection{Single Maximal Tail}

\begin{thm}\label{characterize prim of O_G,Lambda single maximal tail}
Suppose that $k<\infty$ and that $\Lambda$ has only one maximal tail. Then the map $\Phi$ from Theorem~\ref{T:primcha} is a homeomorphism.
\end{thm}
\begin{proof}
First notice that $M(\Lambda)=\{\Lambda^0\}$ as $\Lambda$ has only one maximal tail.
Hence $\Phi$ in \eqref{E:Phi} is a bijection from $\widehat{\Per}_{G,H_{\Lambda^0} \Lambda}$ onto $\mathrm{Prim}(\mathcal{O}_{G,\Lambda})$. Fix a subset $D \subseteq \widehat{\Per}_{G,H_{\Lambda^0} \Lambda}$. It suffices to show that $D$ is closed if and only if $\Phi(D)=\{I_{\Lambda^0,f}:f \in D\}$ is closed. Suppose that $D$ is closed. For any $I_{\Lambda^0,g} \in \overline{\Phi(D)}$, we have $\bigcap_{f \in D}I_{\Lambda^0,f}\subseteq I_{\Lambda^0,g}$. By Corollary~\ref{C:con1}, $g \in D$, and so $I_{\Lambda^0,g} \in \Phi(D)$. Thus $\Phi(D)$ is closed. Conversely, suppose that $\Phi(D)$ is closed. For any $g \in \overline{D}$, by Corollary~\ref{C:con1}, $\bigcap_{f \in D}I_{\Lambda^0,f} \subseteq I_{\Lambda^0,g}$. So $I_{\Lambda^0,g} \in \overline{\Phi(D)}=\Phi(D)$. Hence $g \in D$ as $\Phi$ is a bijection. Therefore $D$ is closed.
\end{proof}

\begin{eg}
Suppose that $k<\infty$ and that $\Lambda$ is strongly connected. Then $\Lambda$ has only one maximal tail. So Theorem~\ref{characterize prim of O_G,Lambda single maximal tail} applies.
Also, recall that $H_{\Lambda^0}=\Lambda^0$ if $\Lambda$ is strongly connected (\cite{HLRS15}).
\end{eg}

\begin{eg}[{Product of odometers}]
Let $k<\infty$, $G=\bZ$, and $\Lambda$ a single-vertex $k$-graph. Consider the product of odometers in \cite{LY17}:
\begin{enumerate}
\item $\Lambda^{e_i}:=\{x_{\fs}^i\}_{\fs =0}^{n_i-1},1\le i\le k,n_i>1$;
\item $1\cdot {x_\fs^i}=x_{(\fs+1)\ \text{mod } n_i}^i,1\le i\le k,0\le \fs\le n_i-1$;
\item $1|_{x_\fs^i}= \begin{cases}
    0 &\text{ if }0\le \fs< n_i-1\\
    1 &\text{ if }\fs=n_i-1
\end{cases}
\ (1\le i\le k);$
\item $x^i_\fs x^j_\ft = x^j_{\ft'} x^i_{\fs'} \text{ if } 1\leq i<j \leq k, \fs+\ft n_i=\ft'+\fs' n_j$.
\end{enumerate}
By \cite[Theorem~7.4]{LY18}, $\Per_{G,\Lambda}=\Per_\Lambda=\{p\in \bZ^k:\prod_{i=1}^{k}n_i^{p_i}=1\}$. Let $r$ be the rank of $\Per_\Lambda$. Since $\Lambda$ is strongly connected, by Theorem~\ref{characterize prim of O_G,Lambda single maximal tail}, $\mathrm{Prim}(\mathcal{O}_{G,\Lambda}) \cong \mathbb{T}^r$.
\end{eg}

\subsection{Strongly Aperiodicity}

This subsection is motivated by \cite{KP14}.

\begin{defn}[{\cite[Definition~3.1]{KP14}}]
A $k$-graph $\Lambda$ is said to be \emph{strongly aperiodic}, if for any hereditary and saturated set $H \subsetneq \Lambda^0, \Lambda(\Lambda^0 \setminus H)$ is aperiodic.
\end{defn}

\begin{lem}\label{cycline triple is cycline pair}
For any hereditary and saturated set $H \subsetneq \Lambda^0$, any cycline triple of $(G,\Lambda(\Lambda^0 \setminus H))$ is a cycline pair of $(G,\Lambda(\Lambda^0 \setminus H))$.
\end{lem}
\begin{proof}
Fix a cycline triple $(\mu,g,\nu)$ of $(G,\Lambda(\Lambda^0 \setminus H))$. Take an arbitrary $x \in s(\mu) (\Lambda(\Lambda^0 \setminus H))^\infty$. Let
$T:=\{v \in \Lambda^0:v \Lambda x(p,p)\ne \mt \text{ for some }p\in \bN^k\}$.
Then $T$ is a maximal tail of $\Lambda$ and $(\mu,g,\nu)$ is a cycline triple of $(G,\Lambda T)$.
%We assumed in the beginning of this section that any cycline triple of $(G,\Lambda T)$
From Condition~(Cyc), $(\mu,g,\nu)$ is a cycline pair of $(G,\Lambda T)$. So $g=1_G$ and we are done.
\end{proof}

The following lemma strengthens \cite[Proposition~3.3]{KP14}.

\begin{lem}\label{L:eqape}
The following statements are equivalent.
\begin{enumeratei}
\item
%\label{Lambda strongly ape}
$\Lambda$ is strongly aperiodic;
\item
%\label{any H, Per H=0}
$\Per_{G,\Lambda (\Lambda^0 \setminus H)}=\{0\}$ for any hereditary and saturated set $H \subsetneq \Lambda^0$;
\item
%\label{any T, Per T=0}
$\Per_{G,\Lambda T}=\{0\}$ for any maximal tail $T$;
\item
%\label{primitive ideal is I(Lambda^0 setminus T)}
$\mathrm{Prim}(\mathcal{O}_{G,\Lambda})=\{T:T \in M_{\gamma}(\Lambda)\}$;
\item
%\label{ideal is I(H)}
every ideal of $\mathcal{O}_{G,\Lambda}$ is of the form $I(H)$ for some hereditary and saturated set $H$;
\item
%\label{every ideal is gauge-inv}
every ideal of $\mathcal{O}_{G,\Lambda}$ is gauge-invariant and diagonal-invariant.
\end{enumeratei}
\end{lem}

\begin{proof}
(i)$\iff$(ii). It follows from \cite[Theorem~4.2]{LY18} and Lemma~\ref{cycline triple is cycline pair}.

(ii)$\implies$(iii). It is straightforward.

(iii)$\implies$(ii). Fix a hereditary and saturated set $H \subsetneq \Lambda^0$. By Lemma~\ref{cycline triple is cycline pair}, we only need to show that every cycline pair of $(G,\Lambda(\Lambda^0 \setminus H))$ is trivial. Fix a cycline pair $(\mu,1_G,\nu)$ of $(G,\Lambda(\Lambda^0 \setminus H))$. We apply the same argument from Lemma~\ref{cycline triple is cycline pair} here. Take an arbitrary $x \in s(\mu) (\Lambda(\Lambda^0 \setminus H))^\infty$. Let $T:=\{v \in \Lambda^0:v \Lambda x(p,p)\ne \mt \text{ for some }p\in \bN^k\}$. Then $T$ is a maximal tail of $\Lambda$ and $(\mu,1_G,\nu)$ is a cycline pair of $(G,\Lambda T)$. Since $\Per_{G,\Lambda T}=\{0\}$, we have $\mu=\nu$.

(iii)$\implies$(iv). It follows from Theorem~\ref{T:primcha} and Corollary~\ref{C:con1}.

(iv)$\implies$(v). Fix an ideal $I$ of $\mathcal{O}_{G,\Lambda}$. By the assumption of (iv) and by \cite[Proposition~A.17]{RW98}, we can find a subset $Y \subseteq M_\gamma(\Lambda)$ such that $I=\bigcap_{T \in Y}I_T$. Notice that for each $T \in Y$, $I_T=I(\Lambda^0 \setminus T)$ due to Corollary~\ref{C:con1}. Thus $I_T$ is both gauge-invariant and diagonal-invariant by Theorem~\ref{T:chagi}, and so is $I$. Therefore Thus $I=I(H(I))$ (in fact, $H(I)=\Lambda^0 \setminus \bigcup_{T \in Y}T$).

(v)$\implies$(vi). It is straightforward.

(vi)$\implies$(iii). Suppose that there exists a periodic maximal tail $T$, that is, $\Per_{G,\Lambda T}\neq\{0\}$. Choose an arbitrary $z \in \Per_{G,\Lambda T} \setminus \{0\}$ and an arbitrary $f \in \widehat{\Per}_{G,\Lambda T}$ such that $f(z) \neq 1$. Write $z=p-q$ for some $p,q \in \mathbb{N}^k$. Find $\mu \in \Lambda^p$ and $\nu \in \Lambda^q$ such that $(\mu,1_G,\nu)$ a cycline pair of $\Lambda T$. Notice that $s_\mu-f(z)s_\nu \in I_{T,f}$. Pick up an arbitrary $\lambda \in \mathbb{T}^k$ such that $\lambda^{d(\mu)} \neq \lambda^{d(\nu)}$. Then we observe that $\gamma_\lambda(s_\mu-f(z)s_\nu) \notin I_{T,f}$. So $I_{T,f}$ is not gauge-invariant, which is a contradiction. Hence $\Per_{G,\Lambda T}=\{0\}$.
\end{proof}

The following theorem is a generalization of \cite[Theorem~3.15]{KP14}.

\begin{thm}\label{T:chasape}
Suppose that $\Lambda$ is strongly aperiodic. Let $\mt\ne Y\subseteq M(\Lambda)$ and $T_0 \in M(\Lambda)$. Then $T_0 \in \overline{Y}\iff T_0 \subseteq \bigcup_{T \in Y}T$.
\end{thm}

\begin{proof}
This follows from
\begin{align*}
T_0 \in \overline{Y}
&\iff \bigcap_{T \in Y}T \subseteq T_0\\
&\iff \bigcap_{T \in Y}I(\Lambda^0 \setminus T) \subseteq I(\Lambda^0 \setminus T_0)\ (\text{by Corollary~\ref{C:con1} and Lemma~\ref{L:eqape}})\\
&\iff T_0 \subseteq \bigcup_{T \in Y}T\ (\text{by Theorem~\ref{T:chagi}}).
\end{align*}
\end{proof}

\end{document}